\documentclass[12pt]{article}
\usepackage{amsmath,amssymb,amsthm}
\usepackage{graphics,epsfig}
\usepackage{hyperref}
\usepackage{natbib}
\usepackage{color}
\usepackage{graphicx}
\usepackage{caption}
\usepackage{subcaption}
\usepackage{float}
\usepackage{mathrsfs}
\usepackage{multirow}
\usepackage{comment}
\usepackage{natbib}
\usepackage{setspace}
\usepackage{geometry}
\usepackage{bm}

\def \ra {\rightarrow}

\def \E {\mathbb{E}}

\def \a {\alpha}
\def \be {\beta}

\newtheorem{definition}{\bf Definition}

	\newtheorem{remark}{\bf Remark}

	\newtheorem{theorem}{\bf Theorem}
	
	\newtheorem{prop}[theorem]{\bf Proposition}
	\newtheorem{lem}[theorem]{\bf Lemma}

	\newtheorem{as}[theorem]{\bf Assumption}

\begin{document}
	
\title{\bfseries Optimal Feedback Control in Social Networks in a McKean-Vlasov-Friedkin-Johnsen System}

\author{
	Paramahansa Pramanik\footnote{Corresponding author, {\small\texttt{ppramanik@southalabama.edu}}}\; \footnote{Department of Mathematics and Statistics,  University of South Alabama, Mobile, AL, 36688,
		United States.}
}

\date{\today}
\maketitle

\begin{abstract}
This paper presents a comprehensive analytical formulation for deriving a closed-form optimal strategy for agents operating within a social network, modeled through a  McKean-Vlasov stochastic differential equation (SDE). Each agent aims to minimize a personal dynamic cost functional that accounts for deviations from the collective opinions of others, their own past beliefs, and is influenced by randomness and inherent opinion rigidity, often described as stubbornness. To tackle this, we develop a novel methodology rooted in a Feynman-type path integral framework, incorporating a specially designed integrating factor to obtain explicit feedback control laws. This approach provides a tractable and insightful solution to the control problem in a setting shaped by both memory and noise. As part of our analysis, we adopt a modified form of the Friedkin-Johnsen opinion dynamics model to more accurately capture the influence of prior beliefs and social interactions, enabling the explicit derivation of the optimal strategy. Comparative simulations further illustrate the effectiveness and adaptability of our method across different network structures, highlighting its potential relevance to understanding opinion evolution and influence strategies in complex social systems.
\end{abstract}

\subparagraph{Key words:}
Opinion dynamics; McKean-Vlasov SDE; path integral control; Friedkin-Johnsen model.

\section{Introduction.}
In recent years, the literature of opinion dynamics has garnered considerable attention, primarily due to its ability to model and explain the evolution of individual and collective beliefs within interconnected social systems \citep{acemoglu2011opinion}. This domain focuses on how a population’s beliefs, often represented as scalar values or multidimensional vectors, shift and adapt over time through ongoing interactions among agents within a network. These belief adjustments are frequently modeled as movements toward a weighted average of others’ opinions, thereby reflecting the influence of social interactions and relational structure \citep{pramanik2021consensus}. Classic continuous opinion dynamics models often assume that, in a connected network, the interplay of influence across agents inevitably leads to a consensus, an eventual alignment of beliefs across the entire population \citep{stella2013opinion}. Yet, this theoretical consensus is challenged under alternative modeling frameworks such as bounded confidence models, where agents selectively interact only with others whose beliefs lie within a pre-defined tolerance band. These models highlight how real-world belief formation is constrained by cognitive and psychological biases, leading individuals to ignore opinions that differ too greatly from their own. Furthermore, another significant deviation from consensus models emerges when agents possess a certain level of stubbornness, rendering them resistant to external influence. These ``stubborn agents" may function as political leaders, media outlets, or ideologically rigid individuals who maintain fixed beliefs while continuing to exert influence on the broader network \citep{acemouglu2013opinion}. Such models are not only more reflective of real-world phenomena but also introduce mathematical complexity by inhibiting full convergence and fostering persistent belief heterogeneity.

As opinion dynamics models scale to larger populations, researchers have turned to the study of mean-field limits to understand emergent behavior in the thermodynamic limit. Specifically, in homogeneous agent settings, as the number of interacting individuals grows, the stochastic evolution of their opinions tends to converge toward a deterministic process described by a mean-field partial differential equation defined over the space of probability measures. This convergence is consistent with the propagation of chaos phenomenon observed in interacting particle systems, where the behavior of any finite subset of agents becomes asymptotically independent from the rest, conditioned on the mean-field \citep{stella2013opinion}. These results provide powerful analytical tools for approximating the macroscopic evolution of beliefs from microscopic interaction rules. Empirical and simulation-based studies further validate these theoretical outcomes by demonstrating how large-scale social influence often results in patterns such as consensus, polarization, or fragmentation, depending on interaction topology and initial opinion distribution \citep{castellano2009statistical}. Notably, global or uniform interactions often facilitate consensus, while localized interactions tend to generate opinion clusters, reflecting ideological segmentation commonly observed in social networks. This segmentation arises when agents are limited to engaging with like-minded peers, leading to groupings of similar beliefs and the marginalization of dissenting voices. To further analyze these phenomena, researchers have developed sophisticated models such as the Lagrangian framework for dissensus, which utilizes graph-theoretic constructs and stochastic stability analysis to examine how structural properties of networks and noise levels influence the persistence of divergent beliefs \citep{bauso2016opinion}. These theoretical advancements contribute significantly to our understanding of how opinions evolve, stabilize, or diverge in complex social systems, offering valuable insights for applications ranging from political campaign strategies to information dissemination and collective decision-making.

Social networks play a fundamental role in shaping a wide range of individual behaviors and socioeconomic outcomes. Empirical research has consistently demonstrated that network structures significantly impact educational achievements \citep{calvo2009}, access to employment opportunities \citep{calvo2004}, adoption of new technologies \citep{conley2010}, consumer behavior \citep{moretti2011}, and even health-related habits such as smoking \citep{nakajima2007,sheng2020}. These pervasive influences stem from the fact that social networks are not exogenously given but are endogenously formed through the decisions of individual agents. As such, a thorough understanding of consensus the tendency of beliefs or actions to align across agents is essential to uncovering the mechanisms behind network formation and its broader implications. While the theoretical underpinnings of social networks have been extensively explored, the specific treatment of consensus formation as a Nash equilibrium in stochastic networks remains underdeveloped in the literature. In addressing this gap, \citet{sheng2020} conceptualizes network formation as a simultaneous-move game in which agents strategically choose their social ties based on the utility derived from both direct and indirect connections. This framework reflects the endogenous nature of network development and aligns with rational choice theory, which posits that individuals form links not arbitrarily but to maximize their expected payoffs. Furthermore, \citet{sheng2020} contributes to the empirical analysis of social networks by proposing a method that allows for the partial identification of large-scale network structures in a computationally tractable manner, thereby overcoming some of the scalability limitations often associated with network inference.

The empirical study of network formation has deep historical roots, beginning with the seminal work of \citet{erdos1959}, who introduced the concept of random graphs composed of independently formed links with a constant probability. The Erdős--Rényi model laid the groundwork for probabilistic modeling of networks, but its simplistic assumptions have spurred the development of more complex models capable of generating graphs that resemble real-world networks. These alternative models aim to capture features such as heterogeneous degree distributions, clustering, and short average path lengths, hallmarks of empirical social networks often described as small-world or scale-free. In particular, model-based approaches are considered valuable not only for simulating realistic network topologies but also for enabling empirical analysis, provided that the models are amenable to estimation and inference. One widely adopted framework in this regard is the exponential random graph model (ERGM), which can be parameterized to match observed network statistics and thus reproduce key structural characteristics of empirical networks \citep{snijders2002,hua2019,polansky2021motif}. Despite its flexibility, ERGM suffers from a lack of microfoundations, making it ill-suited for conducting counterfactual analysis, a major limitation from an economic perspective, where individual decision-making and utility maximization are central assumptions. To address this, alternative approaches have been proposed that treat networks as outcomes of stochastic processes governed by underlying probabilistic rules. These process-based models shift the analytical focus from individual agents to the evolution of network structures over time, emphasizing the estimation of parameters that define the transition dynamics rather than the properties of specific network realizations \citep{polansky2021motif}. Together, these modeling paradigms underscore the importance of integrating both structural realism and theoretical rigor in the study of social networks, particularly when exploring equilibrium behavior such as consensus formation within dynamic and stochastic environments.

Figure \ref{fig:ergm_combined} comprises two distinct visual representations designed to illustrate the structural and dynamical properties of ERGMs under the influence of stochastic opinion dynamics. The left subfigure displays a densely connected network generated from an ERGM, with opinion evolution governed by a stochastic McKean-Vlasov SDE incorporating Friedkin-Johnsen-type feedback. Each vertex represents an agent whose opinion evolves as a solution to a controlled stochastic differential equation. These dynamics are shaped by three principal forces: the statistical mean-field arising from the population's distribution, memory of historical beliefs, and direct local interactions with neighboring agents. The presence of heterogeneity is explicitly encoded in the graph: vertices are visually differentiated by color and size, where node color distinguishes standard agents from stubborn or opinion-fixed nodes, and node size may encode influence weight or degree centrality. Edges denote bilateral influence channels and are assumed to be undirected for modeling symmetric communication or interaction effects.

This particular realization reflects a regime in which central nodes are densely interconnected, forming a core-periphery structure. The inner core facilitates rapid belief convergence among centrally located agents, while peripheral nodes interact less frequently and may exhibit delayed convergence or even persistent disagreement. The observable topology is consistent with known statistical properties of ERGMs designed to match empirical network statistics, such as clustering coefficients and degree distributions.

\begin{figure}[H]
	\centering
	\includegraphics[width=\textwidth]{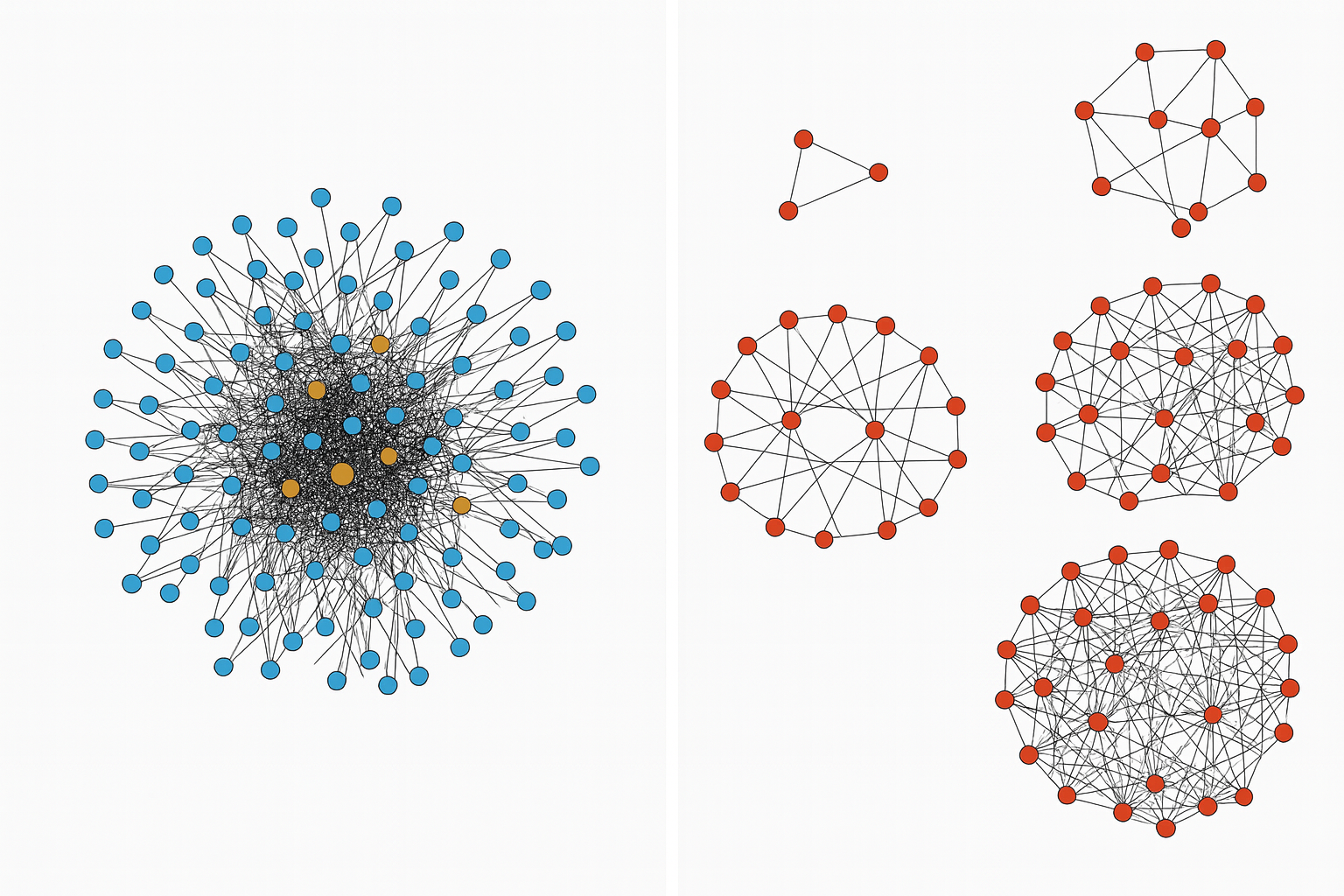}
	\caption{The left panel illustrates a dense network generated under McKean-Vlasov-Friedkin-Johnsen dynamics, where node size and color reflect opinion influence and stubbornness. The right panel shows four smaller ERGM networks with increasing numbers of vertices, demonstrating the structural and topological evolution of social interaction networks.}
	\label{fig:ergm_combined}
\end{figure}

\noindent The coupling of McKean--Vlasov feedback with Fredkin--Johnsen self-reference results in a nontrivial equilibrium state, where agents may stabilize at distinct belief values depending on their topological position and stubbornness parameters. The network thus captures a departure from classical consensus dynamics, admitting stable opinion diversity, clustering, or metastable polarization, depending on the stochastic inputs and influence weights.

The right subfigure consists of four smaller network diagrams, arranged in a 2-by-2 panel format, exhibiting the behavior of ERGM structures under increasing vertex cardinality. The networks represent sequential graph realizations, beginning with minimal configurations and scaling to more complex formations \citep{pramanik2021scoring}. The top-left panel contains a low-order network with minimal connectivity, emphasizing sparse structure and limited paths for opinion diffusion \citep{pramanik2021optimala}. As vertex count increases across the subsequent panels, the networks exhibit growing average degree, tighter clustering, and enhanced path redundancy. The bottom-right network reaches a structural complexity reminiscent of the left subfigure, suggesting a continuum from micro-level interaction systems to large-scale social dynamics. This progression serves to highlight how topological features scale with system size in ERGM constructions and how such scaling interacts with the emergent properties of stochastic opinion processes \citep{pramanik2024parametric}.

Figure \ref{fig:fj_clustered} describes two large-scale network realizations, each comprising over one hundred vertices, constructed under a Friedkin–Johnsen-type opinion dynamics framework. The graphs exhibit a modular structure characterized by multiple distinguishable clusters, where intra-cluster edges are denser than inter-cluster links. Within each panel, nodes are visually differentiated by color: blue nodes represent adaptable agents whose opinions evolve through social influence, while orange nodes correspond to stubborn individuals whose opinions remain largely fixed \citep{pramanik2020optimization,pramanik2023semicooperation}. The structural design mirrors the heterogeneity of real-world social systems, in which both open-minded and resistant agents coexist, influencing the flow and stabilization of collective beliefs. Each cluster can be interpreted as a localized opinion community, shaped by internal reinforcement and partially isolated from external perspectives due to sparse cross-cluster connectivity \citep{pramanik2022lock}.
\begin{figure}[H]
	\centering
	\includegraphics[width=\textwidth]{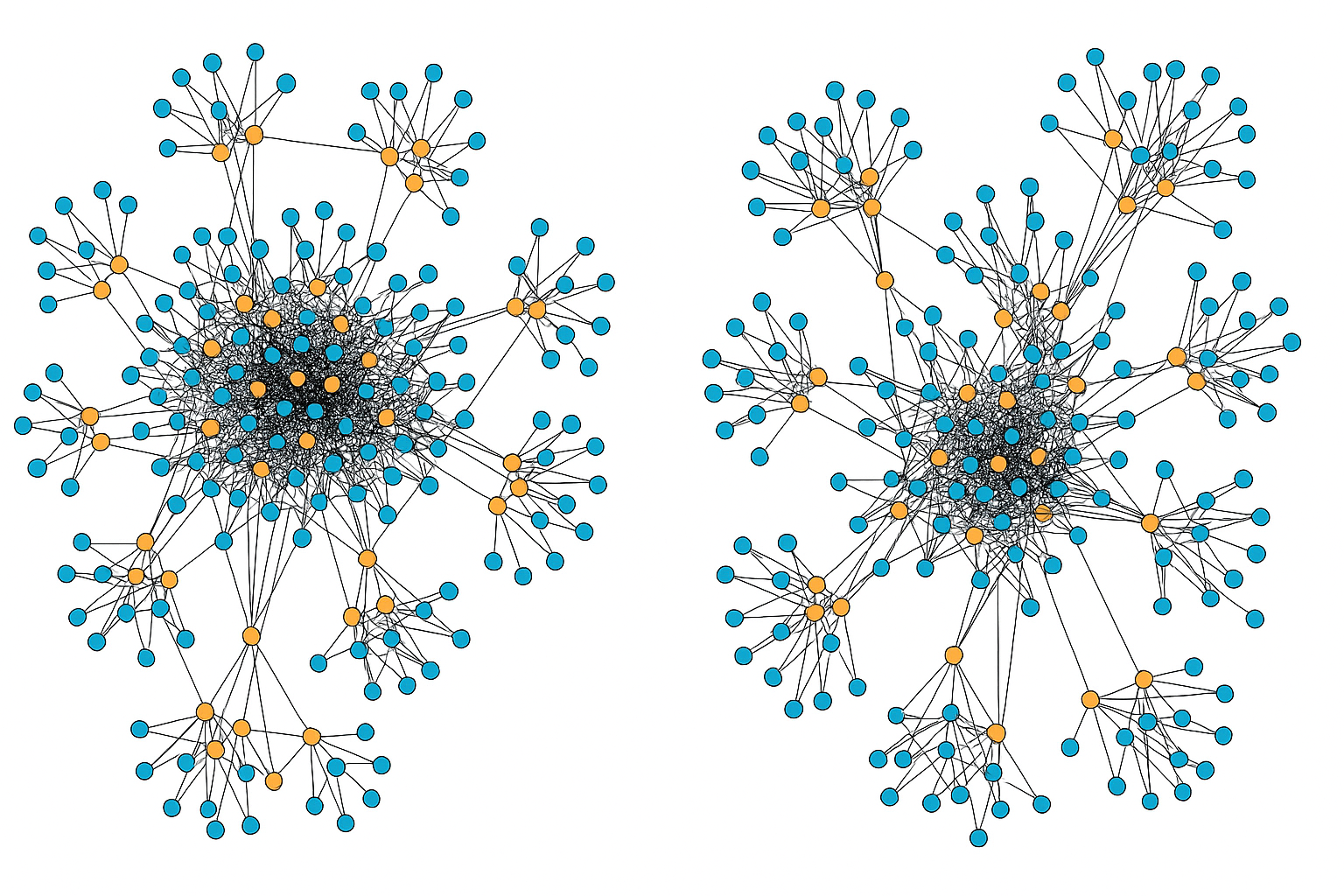}
	\caption{Clustered network topologies generated under the Friedkin-Johnsen opinion dynamics model with over 100 vertices per panel. Each node represents an agent, with orange nodes indicating stubborn agents (resistant to opinion change) and blue nodes representing adaptive agents. The two panels display distinct modular structures with varying inter- and intra-cluster connectivity, illustrating the emergence of opinion clusters due to local influence, stubbornness, and memory-driven dynamics. }
	\label{fig:fj_clustered}
\end{figure}

These network configurations illustrate how combining stochastic interaction models with structured topology can yield sustained opinion diversity and prevent convergence to a single consensus \citep{vikramdeo2023profiling,vikramdeo2024abstract}. The presence of multiple stubborn agents, embedded within various clusters, introduces localized anchoring effects that inhibit homogenization across the global network. This phenomenon aligns with theoretical predictions from bounded confidence and resistance-to-influence models, wherein subpopulations retain distinct viewpoints over time \citep{pramanik2020motivation}. From a statistical modeling perspective, such networks highlight the importance of incorporating both individual-level behavioral rules and macro-level structural features. The interplay between node-level properties and graph topology underscores the need for flexible generative models such as ERGMs conditioned on latent clustering to capture emergent properties like fragmentation, polarization, and partial consensus \citep{pramanik2024estimation,vikramdeo2024mitochondrial,pramanik2024motivation}.

This paper introduces a novel perspective on modeling opinion dynamics by embedding the process within the mathematical structure of McKean-Vlasov dynamics. We investigate a stochastic formulation where beliefs evolve continuously over time in a system comprising a homogeneous group of interacting agents. Foundational work in this domain was conducted by \citet{mckean1967propagation} and \citet{kac1956foundations}, who explored SDEs devoid of control mechanisms to establish results pertaining to propagation of chaos. Kac, in particular, analyzed mean-field SDEs influenced by classical Brownian motion, laying the groundwork for deeper interpretations of the Boltzmann and Vlasov kinetic frameworks. Substantial theoretical development was also contributed by \citet{sznitman1991topics}, whose insights further clarified the probabilistic underpinnings of such systems. Although early progress laid the foundations, the last decade has seen intensified exploration, largely driven by the emergence of mean-field game (MFG) theory independently formulated by \citet{lasry2007mean} and \citet{huang2003individual}. Within this theoretical construct, the McKean–Vlasov equation serves as a core tool to model equilibrium behavior among a continuum of indistinguishable agents whose strategic decisions depend on the aggregate distribution of their states. While MFG theory captures the Nash equilibrium arising from decentralized optimization \citep{pramanik2024bayes,bulls2025assessing}, scenarios involving collective welfare or social efficiency instead consider stochastic control of McKean-Vlasov SDEs, a distinction thoroughly examined by \citet{carmona2013control}.

As supported by prior research, the large-agent limit where the number of interacting participants approaches infinity implies that agents' belief trajectories become asymptotically independent \citep{pramanik2023cont,pramanik2024estimation}. In this regime, each agent's belief or state evolves according to a stochastic differential equation whose coefficients depend not on fixed parameters but on the empirical distribution of private states across the population. Consequently, solving the agent's optimization problem under such dynamics constitutes a stochastic control problem governed by McKean-Vlasov equations, a topic that remains only partially explored and analytically challenging. For an early treatment of this issue, see the contributions of \citet{andersson2011maximum}. An open question within this setting is whether the optimal feedback control derived from this mean-field-type problem yields an approximate equilibrium for the original finite-agent game \citep{pramanik2023path}. Addressing this requires a careful examination of how the optimal feedback strategies from the stochastic control framework correspond to, or diverge from, those derived in classical mean-field game theory \citep{pramanik2023cmbp,pramanik2023optimization001}. Our work contributes to this growing area by investigating whether such feedback strategies not only stabilize individual dynamics but also approximate equilibrium behavior in high-dimensional agent-based systems.

Figure \ref{fig:opinion_dynamics} presents a simulated visualization of opinion dynamics driven by McKean-Vlasov SDEs in a homogeneous agent population. The left panel displays the time evolution of individual opinions for 100 agents, each following a trajectory influenced by both mean-field interaction and stochastic perturbations \citep{kakkat2023cardiovascular,khan2023myb}. These trajectories illustrate how agents gradually adjust their beliefs over time, with most converging toward a shared consensus region while retaining some dispersion due to noise \citep{pramanik2022stochastic,pramanik2024stochastic}. The right panel shows kernel density estimates of the empirical distribution of opinions at several time snapshots, capturing the transition from a wide initial spread to a progressively concentrated distribution \citep{pramanik2021,pramanik2021consensus}. This reflects the gradual alignment of beliefs across the population, driven by the agents' attraction to the population mean and modulated by stochastic fluctuations. Together, both panels demonstrate how the McKean-Vlasov opinion dynamics the emergence of consensus and the role of randomness in shaping belief dynamics in large interacting systems.
\begin{figure}[H]
	\centering
	\includegraphics[width=\textwidth]{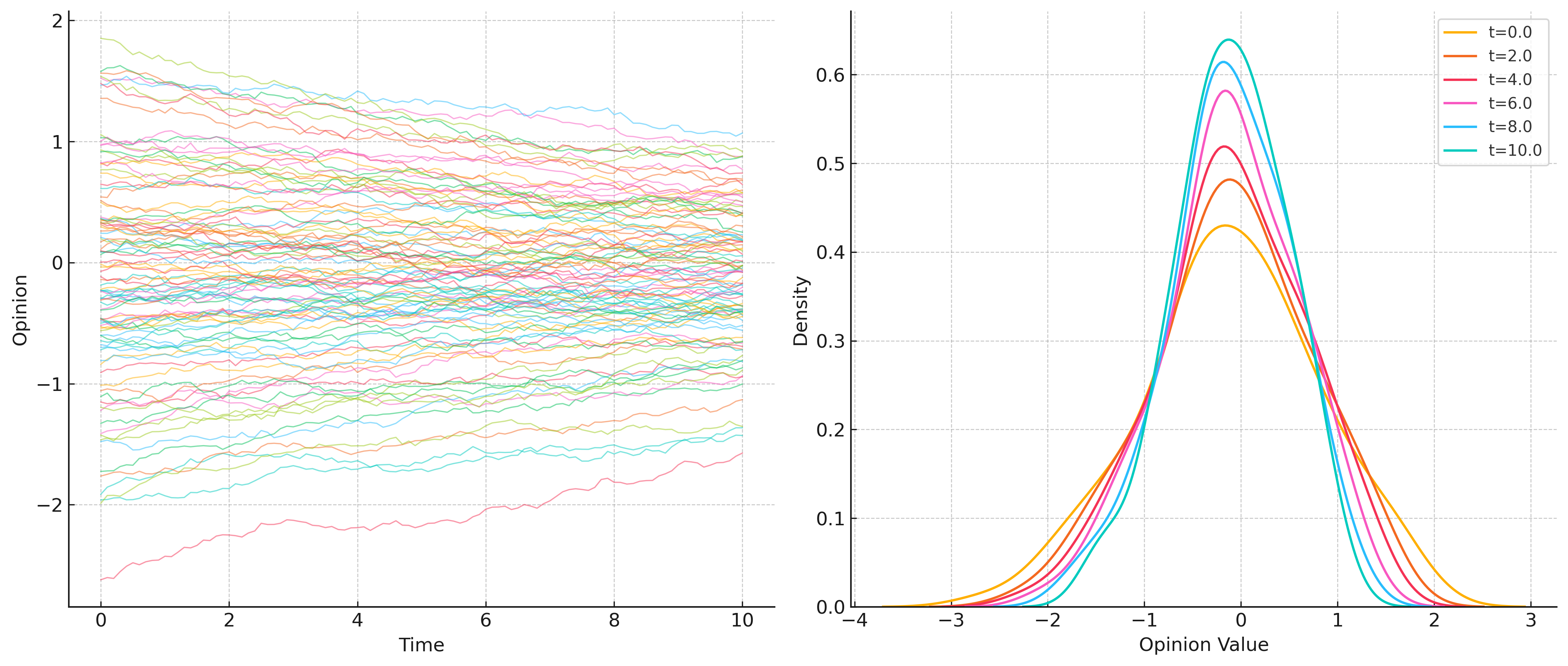}
	\caption{Simulated opinion evolution under McKean-Vlasov dynamics. Left: individual agent opinion trajectories over time. Right: density estimates of opinion distributions at selected time points.}
	\label{fig:opinion_dynamics}
\end{figure}

A particular class of nonlinear Hamilton-Jacobi-Bellman (HJB) equations can be reformulated into linear form by applying a logarithmic transformation an approach that originates from early developments in quantum mechanics, where Schr\"odinger first established a connection between the HJB approach and his eponymous equation \citep{dasgupta2023frequent,hertweck2023clinicopathological,khan2024mp60}. This linearization enables one to replace the backward time integration typically required for solving HJB equations with an alternative forward-time computation based on expected values derived from a stochastic process. These expectations correspond to averages over sample paths defined by a forward diffusion, which can be expressed through a path integral representation. In certain complex systems, such as those governed by the Merton-Garman-Hamiltonian, applying the Pontryagin Maximum Principle becomes analytically intractable, making the Feynman path integral method a practical alternative. This technique has been successfully applied in various domains. For instance, it has been utilized in motor control theory as demonstrated by \citet{kappen2005}, \citet{theodorou2010}, and \citet{theodorou2011}, and has also seen extensive use in quantitative finance, as detailed by \citet{belal2007}. Moreover, \citet{pramanik2020} proposed a Feynman-inspired path integral formulation to derive optimal feedback controls, and a broader generalization of Nash equilibrium concepts using tensor field structures was explored in \citep{pramanik2023semicooperation}.

The organization of the paper is as follows: Section 2 introduces the mathematical formulation of the opinion dynamics model along with the associated cost functional. Section 3 presents the key assumptions and fundamental characteristics of the stochastic McKean-Vlasov framework. In Section 4, we derive the system’s deterministic Hamiltonian structure and its corresponding stochastic Lagrangian representation. Section 5 details the central theoretical contributions involving Feynman-type path integral control, with a specific focus on its implementation within the Fredkin-Johnsen model context. Lastly, Section 6 summarizes the findings and outlines potential directions for future research.

\section{Opinion Dynamics as a Differential Game.}

Following \cite{niazi2016} consider a social network of $n$ agents by a weighted directed graph $G=(N,E,w_{ij})$, where $N=\{1,...,n\}$ is the set of all agents. Let, $E\subseteq N\times N$ be the set of all ordered pairs of all connected agents, and $w_{ij}$ be the influence of agent $j$ on agent $i$ for all $(i,j)\in E$. There are two types of connections, one-sided or two-sided. For the principle-agent problem, the connection is one-sided (i.e. Stackelberg model), and for the agent-agent problem, it is two-sided (i.e. Cournot model). Suppose $x^i(s)\in[0,1]$ be the opinion of agent $i^{th}$ at time $s\in[0,t]$ with their initial opinion $x^i(0)=x_0^i\in[0,1]$. Then $x^i(s)$ has been normalized into $[0,1]$ where $x^i(s)=0$ stands for a strong disagreement and $x^i(s)=1$ represents strong agreement and all other agreements stay in between. Let $\mathbf x(s)=\left[x^1(s),x^2(s),...,x^n(s)\right]^T\in[0,1]^n$ be the opinion profile vector of $n$ agents at time $s$ where T represents the transposition of a vector. Following \cite{niazi2016} define the cost function of agent $i$ as
\begin{align}\label{0}
L^i(s,\mathbf x,u^i)&:= \E\left\{\mbox{$\frac{1}{2}$}\int_0^t \bigg\{\sum_{j\in\eta_i}w_{ij}\left[x^i(s)-x^j(s)\right]^2+k_i\left[x^i(s)-x_0^i\right]^2+\left[u^i(s)\right]^2\bigg\}ds\right\},
\end{align}
where $w_{ij}\in[0,\infty)$ is a parameter which weighs the susceptibility of agent $j$ to influence agent $i$, $k_i\in[0,\infty)$ is agent $i$'s stubbornness, $u^i(s)\in\mathcal U([0,t])$ is an adaptive control process of agent $i$ taking values in a convex open set in $\mathbb R^n$, and set of all agents with whom $i$ interacts is $\eta_i$ and defined as $\eta_i:=\{j\in N:(i,j)\in E\}$. In this paper $u^i$ represents agent $i$'s control over their own opinion as well as influencing other agents' opinions. The cost function $L^i(s,\mathbf x,u^i)$ is twice differentiable with respect to time in order to satisfy Wick rotation, is continuously differentiable with respect to $i^{th}$ agent's control $u^i(s)$, non-decreasing in opinion $x^i(s)$, non-increasing in $u^i(s)$, and convex and continuous in all opinions and controls \citep{mas1995,pramanik2023path}. 

\begin{figure}[ht!]
	\centering
	\includegraphics[width=\textwidth]{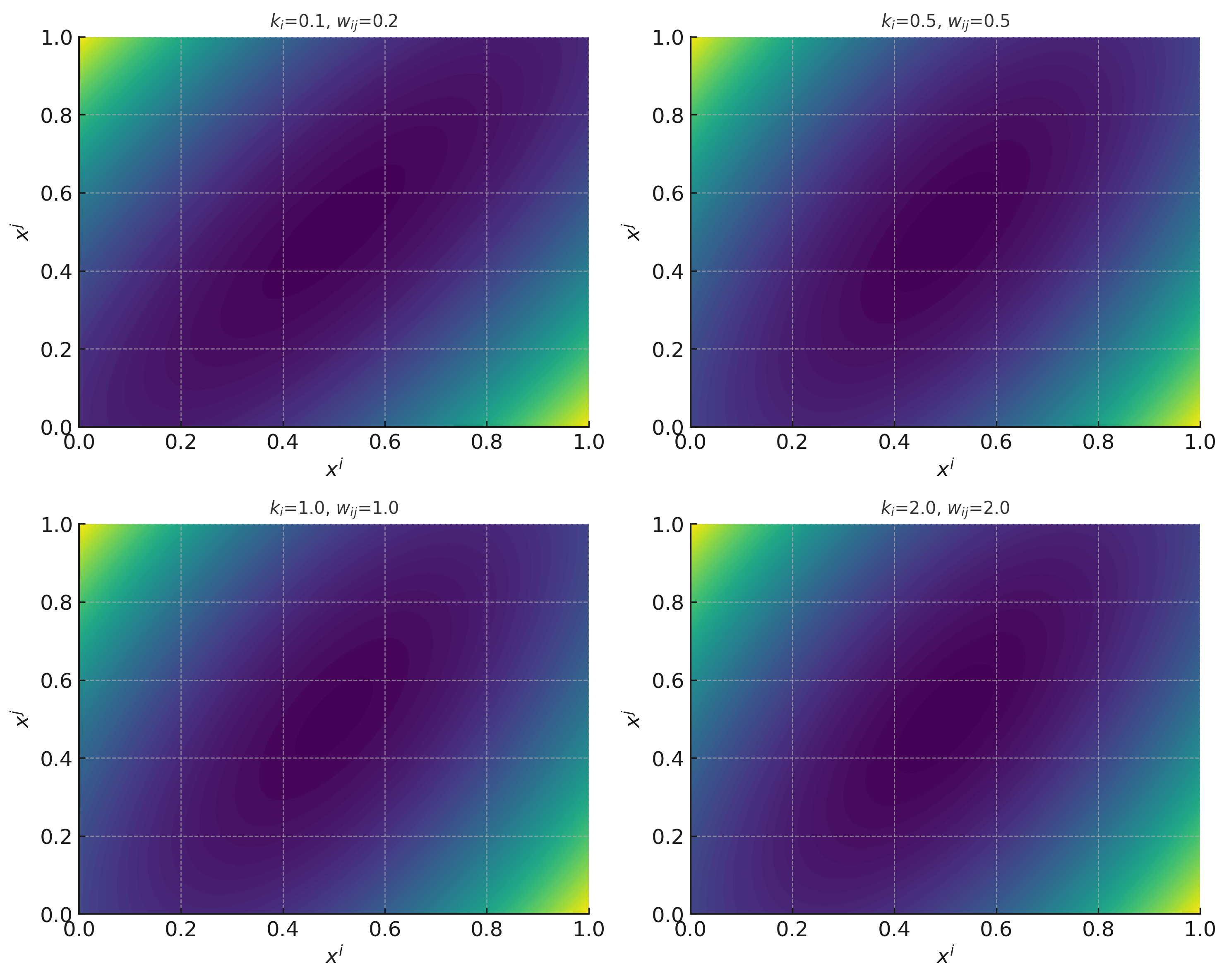}
	\caption{Contour plots of the agent’s cost function $L^i$ over the opinion space $(x^i, x^j)$ for varying combinations of stubbornness $k_i$ and influence weight $w_{ij}$. All plots assume a fixed control input $u^i = 0.5$. As $k_i$ increases, the cost surface becomes more sensitive to deviations from the agent's initial belief. Higher $w_{ij}$ values increase the penalty for disagreement between connected agents, steepening the cost landscape in the $x^j$ direction. These panels illustrate how opinion coupling and stubbornness shape the optimization landscape in decentralized decision-making models.}
	\label{fig:cost_contour_panels}
\end{figure}

Figure~\ref{fig:cost_contour_panels} presents a series of contour plots depicting the cost function $L^i$ (derived in Equation \eqref{0}) over the opinion space $(x^i, x^j)$ under varying values of the stubbornness parameter $k_i$ and influence weight $w_{ij}$ \citep{pramanik2025stubbornness,pramanik2025factors}. In each panel, the control input $u^i$ is fixed to a constant value, and the agent's cost is evaluated based on its deviation from both its own initial belief and the opinion of a connected agent \citep{pramanik2025optimal}. As $k_i$ increases, the plots reveal greater curvature along the $x^i$-axis, indicating that the agent increasingly penalizes deviations from its prior belief \citep{pramanik2016,pramanik2021thesis}. Similarly, larger values of $w_{ij}$ steepen the cost gradient along the $x^j$-axis, emphasizing the agent's sensitivity to discrepancies with influential neighbors \citep{pramanik2023cont,pramanik2024estimation}. These contour maps highlight how the interplay between intrinsic stubbornness and social influence defines the landscape over which agents seek to minimize their individual costs in a stochastic control setting \citep{maki2025new}.

\noindent The opinion dynamics of agent $i$ follow a McKean-Vlasov SDE
\begin{align}\label{1}
dx^i(s)&=\mu^i[s,x^i(s),\mathbb P_{(x^i)},u^i(s)]ds+\sigma^i[s,x^i(s),\mathbb P_{(x^i)},u^i(s)]dB^i(s),
\end{align}
with the initial condition $x_0^i$, where $\mu^i$ and $\sigma^i$ are the drift and diffusion functions, $\mathbb P_{(x^i)}$ is the probability law the opinion of agent $i$ with Brownian motion $B^i(s)=\left\{B^i(s),s\in[0,t]\right\}$ \citep{pramanik2024dependence}. The reason behind incorporating Brownian motion in agent $i$'s opinion dynamics is because of Hebbian Learning which states that \citep{pramanik2025stubbornness,pramanik2025factors}, neurons increase the synaptic connection strength between them when they are active together simultaneously, and this behavior is probabilistic in the sense that, resource availability from a particular place is random  \citep{hebb2005,kappen2007}. For example, for a given stubbornness, and influence from agent $j$, agent $i$'s opinion dynamics has some randomness in opinion. Suppose from other resources agent $i$ knows that the information provided by agent $j$'s influence is misleading \citep{pramanik2023optimization001,pramanik2025strategies}. Apart from that after considering humans as automatons, motor control and foraging for food become a big example of minimization of costs (or the expected return) \citep{kappen2007}. As control problems like motor controls are stochastic in nature because there is a noise in the relation between the muscle contraction and the actual displacement with joints with the change of the information environment over time, we consider the Feynman path integral approach to determine the stochastic control after assuming the opinion dynamics Equation (\ref{1}) \citep{feynman1949,fujiwara2017}. The coefficient of the control term in Equation (\ref{0}) is normalized to $1$, without loss of generality. The cost functional represented in the Equation (\ref{0}) is viewed as a model of the motive of agent $i$ towards a prevailing social issue \cite{niazi2016}. The aim of this paper is to characterize a feedback Nash equilibrium $u^{i*}\in\mathcal U([0,t])$ such that
\[
L^i(u^{i*})=\underset{u^i\in\mathcal U([0,t])}{\arg\min}\E_s\left\{L^i(s,\mathbf x,u^i)\big|\mathcal F_0^x\right\},
\]
subject to the Equation (\ref{1}), where $\E_0(L^i|\mathcal F_0^x)$ represents the expectation on $L^i$ at time $0$ subject to agent $i$'s opinion filtration $\mathcal F_0^x$ generated by the Brownian motion $B^i$ starting at the initial time $0$ for a complete probability space $(\Omega,\mathcal F,\mathcal F_0^x,\mathbb P)$. A solution to this problem is a feedback Nash equilibrium as the control of agent $i$ is updated based on the opinion at the same time $s$ \citep{pramanik2025impact,pramanik2025strategic}.

\section{Background.}
Let $t>0$ be a fixed finite horizon. Assume  $ B^i(s)=\{ B^i(s)\}_{s=0}^t$ is a 1-dimensional Brownian motion defined on a probability space $(\Omega, \mathcal F,\mathbb P)$, and $\mathcal F_s=\{\mathcal F_s^x\}_{s=0}^t$ be its natural filtration augmented with an independent $\sigma$-algebra $\mathcal F_0^x$, where $\mathbb P$ be the probability law defined above \citep{pramanik2024dependence,pramanik2024measuring}. The McKean-Vlasov stochastic opinion dynamic of agent $i$ is represented in Equation(\ref{1}), where the drift and diffusion coefficients of opinion $x^i(s)$ are given by a pair of deterministic functions $(\mu^i,\sigma^i):[0,t]\times\mathbb R\times\mathcal P_2\left(\mathbb R\right)\times\mathcal U\ra \mathbb R\times\mathbb R$, and $u^i=\{u^i(s)\}_{s=0}^t$ is the \emph{admissible control} of agent $i$ assumed to be a progressively measurable process with values in a measurable space $(\mathcal U, \mathcal U^*)$ \citep{pramanik2024parametric}. In $(\mathcal U, \mathcal U^*)$, $\mathcal U$ is an open subset of an Euclidean space $\mathbb R$, and $\mathcal U^*$ is a $\sigma$-field induced by a Borel $\sigma$-field in the same Euclidean space \citep{carmona2015}. For a metric space $E$, if $\mathcal E$ is its Borel $\sigma$-field, we use $\mathcal P(E)$ as the notation for the set of all probability measures on $(E,\mathcal E)$. We further assume $\mathcal P(E)$ is endowed with the topology of weak convergence \citep{cosso2023}. If E is a Polish space $G$, then for all $r\geq 1$ with metric $d_G$ define
\[
\mathcal P_r(G):=\left\{\gamma\in\mathcal P(G):\int_G d_G(x_0^i,x^i)^r\gamma(d x^i)<\infty\right\},
\]
where $x_0^i\in G$ is arbitrary. For $r\geq 1$ Wasserstein distance $W_r(\gamma,\gamma')$ on $\mathcal P(E)$ define
\begin{multline*}
W_r(\gamma,\gamma'):=\inf\biggr\{\int_{G\times G}d_G(x^i,y^i)^r\pi(d x^i,d y^i):\pi\in\mathcal P(G\times G)\\
 \text{so that}\ \pi(.\times G)=\gamma, \text{and}\ \pi(G\times .)=\gamma'\biggr\}^{\frac{1}{r}},
\end{multline*}
for all $\gamma,\gamma'\in\mathcal P_r(G)$. The space $\left(\mathcal P_r(G),W_r\right)$ is indeed a Polish space \citep{cosso2023}. The term \textit{non-linear}, used to describe the Equation (\ref{1}), does not mean the drift ($\mu^i$) and diffusion  ($\sigma^i$) coefficients are non-linear functions of $X$, but instead, they not only depend on the value of the unknown process $x^i(s)$ but also on its distributions $\mathbb P_{(x^i)}$ \citep{carmona2015}. The set $\mathcal U$ of \emph{admissible controls} $u^i$ as the set of $ \mathcal U$-valued progressively measurable processes $u^i\in\mathcal H^{2,m}$, where $\mathcal H^{2,\tilde m}$ is a Hilbert space
\[
\mathcal H^{2,\tilde m}:=\left\{y^i\in\mathcal H^{0,\tilde m};\ \E\int_0^t|y^i(s)|^2ds<\infty\right\}
\]
with $\mathcal H^{0,\tilde m}$ being the collection of all $\mathbb R^{\tilde m}$-valued progressive measurable processes on $[0,t]$. Let $\mathcal V$ be a sub-$\sigma$-algebra of $\mathcal F$ so that the following assumption holds \citep{hua2019,pramanik2024motivation}.
\begin{as}\label{as0}
	(i). $\mathcal V$ and the filtration $\mathcal F_\infty$ generated by the  Brownian motion of $i^{th}$ agent $ B^i(s)$ are independent.\\
	(ii). $\mathcal V$ is ``rich enough" by means of the following condition:
	\begin{multline*}
	\mathcal P_2(C([0,t],\mathcal H^{2,\tilde m}))=\bigg\{\mathbb P_{(Y)}\ \text{with}\ Y:[0,t]\times\Omega\ra\mathcal H^{2,\tilde m}\ \text{ continuous and}\\
	\mathcal B([0,t])\otimes\mathcal V-\text{measurable process satisfying}\ \E\int_0^t|y^i(s)|^2ds<\infty\bigg\}.
	\end{multline*}
	In other words, for all $\gamma\in\mathcal P_2(C([0,t],\mathcal H^{2,\tilde m}))$ there exists a continuous and Borel $\mathcal B([0,t])\otimes\mathcal V$-process $y^i:[0,t]\times\Omega\ra\mathcal H^{2,\tilde m}$, such that $\E\int_0^t|y^i(s)|^2ds^<\infty$, and Y has the law (distribution) equal to $\gamma$.
\end{as}

\begin{lem}\citep{cosso2023}.\label{lem0}
	Let $\tilde{\mathcal V}$ be another sub-$\sigma$-algebra of $\mathcal F$ on the probability space $(\Omega,\mathcal F,\mathbb P)$. If Assumption \ref{as0} holds then the following statements are equivalent.\\
	(i). There exists a $\tilde{\mathcal V}$-measurable random variable $z^i:\Omega\ra\mathbb R$ with the Uniform distribution $[0,1]$.\\
	(ii). $\tilde{\mathcal V}$ is ``rich enough" by means of the following condition:
	\begin{multline*}
	\mathcal P_2(C([0,t],\mathcal H^{2,\tilde m}))=\bigg\{\mathbb P_{(y^i)}\ \text{with}\ y^i:[0,t]\times\Omega\ra\mathcal H^{2,\tilde m}\ \text{continuous and}\\
	\mathcal B([0,t])\otimes\tilde{\mathcal V}-\text{measurable process satisfying}\ \E\int_0^t|y^i(s)|^2ds<\infty\bigg\}.
	\end{multline*}
\end{lem}

\begin{remark}
	Consider two sets of sub-$\sigma$-algebras $F_1\in\tilde{\mathcal V}$ and $F_2\in\tilde{\mathcal V}$ in the probability space $(\Omega,\mathcal F,\mathbb P)$ so that $\mathbb P(F_1)>0$, and $F_1\subset F_2$ with $0<\mathbb P(F_1)<\mathbb P(F_2)$ (atomless space). Then statements (i) and (ii) in Lemma \ref{lem0} are equivalent.
\end{remark}

\begin{as}\label{as1}
	(i) There exists a linear, unbounded operator $O:\mathcal D(O)\subset\mathcal H^{2,\tilde m}\ra \mathcal H^{2,\tilde m}$  which facilitates a $C_0$-semigroup of pseudo-contractions $\{\exp(sO);s\geq 0$ in $\mathcal H^{2,\tilde m}$.\\
	(ii). The drift $\mu^i$ and the diffusion coefficients $\sigma^i$ are measurable.\\
	(iii). There exists a constant $C^*$ such that
	\begin{align*}
	|\mu^i(s,\gamma,x^i,u^i)-\mu^i(s,\gamma',x^{i'},u^{i'})|&\leq C^*(W_2(\gamma,\gamma')+|x^{i}(s)-x^{i'}(s)|+|u^i(s)-u^{i'}(s)|),\\
	& \text{for all}\ (\gamma,x^i,u^i),(\gamma',x^{i'},u^{i'})\in\mathcal P_2(\mathbb R\times \mathcal U)\times\mathbb R\times \mathcal U,\\
	|\sigma^i(s,\gamma,x^i,u^i)-\sigma(s,\gamma',x^{i'},u^{i'})|&\leq C^*(W_2(\gamma,\gamma')+|x^i(s)-x^{i'}(s)|+|u^i(s)-u^{i'}(s)|),\\ 
	&\text{for all}\ \gamma,\gamma'\in\mathcal P_2(\mathbb R),\\
	|L^i(s,\mathbf x,u^i)-L^i(s,\mathbf x',u^{i'})|&\leq C^*(W_2(\gamma,\gamma')+|x^i(s)-x^{i'}(s)|+|u^i(s)-u^{i'}(s)|),\\
	& \text{for all}\ (\gamma,x^i,u^i),(\gamma',x^{i'},u^{i'})\in\mathcal P_2(\mathbb R\times \mathcal U)\times \mathbb R\times \mathcal U.
	\end{align*}
	(iv). $\sigma^i$ is differentiable in $(\gamma,x^i,u^i)\in\mathcal P_2(\mathbb R\times\mathcal U)\times\mathbb R\times\mathcal U$, and the derivative $\partial_\gamma\sigma^i:\mathcal P_2(\mathbb R\times\mathcal U)\times\mathbb R\times\mathcal U\ra\mathbb R\times\mathcal U$ is bounded Lipschitz continuous. Hence, there exists a positive constant $C^*$ such that
	\begin{align*}
	|\partial_\gamma\sigma^i(s,\gamma,x^i,u^i)|&\leq C^*,\ \text{for all}\ (\gamma,x^i,u^i)\in\mathcal P_2(\mathbb R\times\mathcal U)\times \mathbb R\times\mathcal U,\\
	|\partial_\gamma\sigma^i(s,\gamma,x^i,u^i)-\partial_\gamma\sigma(s,\gamma,x^{i'},u^{i'})|&\leq C^*(W_2(\gamma,\gamma')+|x^i(s)-x^{i'}(s)|+|u^i(s)-u^{i'}(s)|),\\
	&\text{for all}\ (\gamma, x^i,u^i),(\gamma',x^{i'},u^{i'})\in\mathcal P_2(\mathbb R\times\mathcal U)\times\mathbb R\times\mathcal U.  
	\end{align*}
	(v). Consider $\omega^i=\mu^i$ and $L^i$. Therefore, $\omega^i$ is differentiable in $(\gamma,x^i,u^i)\in\mathcal P_2(\mathbb R\times\mathcal U)\times\mathbb R\times\mathcal U$ and the derivatives $\partial_\gamma\omega^i:\mathcal P_2(\mathbb R\times\mathcal U)\times\mathbb R\times\mathcal U\times(\mathbb R\times\mathcal U)\ra \mathbb R\times\mathcal U$,  $\partial_x\omega^i:\mathcal P_2(\mathbb R\times\mathcal U)\times \mathbb R\times\mathcal U\ra\mathbb R$, and $\partial_u\omega^i:\mathcal P_2(\mathbb R\times\mathcal U)\times\mathbb R\times\mathcal U\ra\mathcal U$ are bounded and Lipschitz continuous. Hence, for a given $s\in[0,t]$, there exists a constant $C^*>0$ so that
	\begin{align*}
	|\partial_\gamma\omega^i(\gamma,x^i,u^i)|+|\partial_x\omega^i(\gamma,x^i,u^i)|&+|\partial_u\omega^i(\gamma,x^i,u^i)|\leq C^*, \\
	& \text{for all}\ (\gamma,x^i,u^i)\in\mathcal P_2(\mathbb R\times\mathcal U)\times \mathbb R\times\mathcal U\times(\mathbb R\times\mathcal U),\\
	|\partial_\gamma\omega^i(\gamma,x^i,u^i)-\partial_\gamma\omega^i(\gamma',x^{i'},u^{i'})|&\leq C^*(W_2(\gamma,\gamma')+|x^i-x^{i'}|+|u^i-u^{i'}|),\\
	&\text{for all}\ (\gamma,x^i,u^i),(\gamma',x^{i'},u^{i'})\in\mathcal P_2(\mathbb R \times\mathcal U)\times \mathbb R\times\mathcal U\times(\mathbb R\times\mathcal U),\\
	|\partial_x\omega^i(\gamma,x^i,u^i)-\partial_x\omega^i(\gamma',x^{i'},u^{i'})|&\leq C^*(W_2(\gamma,\gamma')+|x^i-x^{i'}|+|u^i-u^{i'}|),\\
	&\text{for all}\ (\gamma,x^i,u^i),(\gamma',x^{i'},u^{i'})\in\mathcal P_2(\mathbb R\times\mathcal U)\times \mathbb R\times\mathcal U,\\
	|\partial_u\omega^i(\gamma,x^i,u^i)-\partial_u\omega^i(\gamma',x^{i'},u^{i'})|&\leq C^*(W_2(\gamma,\gamma')+|x^i-x^{i'}|+|u^i-u^{i'}|),\\
	&\text{for all}\ (\gamma,x^i,u^i),(\gamma',x^{i'},u^{i'})\in\mathcal P_2(\mathbb R\times\mathcal U)\times \mathbb R\times\mathcal U.
	\end{align*}
\end{as}

\begin{as}\label{as2}
	Under a feedback control structure of a society there exists a measurable function $h^i$ such that $h^i:[0,t]\times C([0,t]):\mathbb{R}\times\mathcal U\ra \mathcal U$ for which $u^i(s)=h^i[x^i(s,u^i)]$ such that Equation (\ref{1}) has a solution.
\end{as}

\begin{as}\label{as3}
	(i). Let $\mathcal Z$ be the set of total knowledge of the entire society. Agent $i$'s knowledge set is $\mathcal Z_i\subset \mathcal Z$ based on their limitations of acquiring new information to enhance knowledge from their society at a given time. This immediately implies set $\mathcal Z_i$ is different for different agents. An agent with younger age has less limitation to acquire new information to make new opinions.\\
	(ii). The initial cost functional of the society is $L_0:[0,t]\times\mathbb R\times\mathcal U\ra\mathbb R$ such that for agent $i$ satisfies $L_0^i\subset L_0$ in Polish space and both of them are concave which is equivalent to Slater condition \citep{marcet2019}.\\
	(iii). For all $u^i(s)\in\mathcal U([0,t])$ there exists $\epsilon>0$ small enough, so that 
	\[
	\E_0\left\{\int_{0}^t\mbox{$\frac{1}{2}$} \bigg\{\sum_{j\in\eta_i}w_{ij}\left[x^i(s)-x^j(s)\right]^2+k_i\left[x^i(s)-x_0^i\right]^2+\left[u^i(s)\right]^2\bigg\}ds\right\}\geq\epsilon; \ \ i\neq j;\ i,j=1,2,...n,
	\]
	where $\E_0\{.\}=\E\{.|x^i_0\}$.
\end{as}

\begin{remark}
	Assumption \ref{as2} guarantees the possibility of at least one fixed point in the knowledge space. It is important to  note that the agent makes decisions based on all available information. Then following Lemma \ref{lem1} shows that the fixed point indeed unique. Assumption \ref{as3} implies that each agent has some initial cost functional $L^i_0$ at the beginning of $[0,t]$, and conditional expected cost functional $\E_0\{L^i\}$ is positive throughout this time interval \citep{pramanik2024measuring}.
\end{remark}	
\begin{lem}\label{lem1}
	Suppose $i^{th}$ agent's initial opinion $x_0^i\in G$ is independent of $B^i(s)$, and $\mu^i$ and $\sigma^i$ satisfy Assumptions \ref{as0} and \ref{as1}. Then there exists a unique solution to opinion dynamics represented by the Equation (\ref{1}) in $\mathcal H^{2,\tilde m}$. Moreover, there exists some positive constant $\hat c$ on time t, Lipschitz constants $\mu^i$ and $\sigma^i$, the unique solution satisfies 
	\[
	\E\left\{\sup_{s\in[0,t]}|x^i(s)|^2\right\}\leq\hat c(1+\E|x_0^i|^2)\exp(\hat c t),
	\]
	for all $i=1,2,...,n$.
\end{lem}
\begin{proof}
	See the Appendix.
\end{proof}	

\begin{remark}
	Lemma \ref{lem1} guarantees that the stochastic opinion dynamics shown in Equation (\ref{1}) exhibits a unique fixed point and the expectation is bounded in the polish space G.
\end{remark}

Assume the set of admissible strategies $\mathcal U([0,t])$ is convex and $u^i\in\mathcal U([0,t])$. Define $x^{i*}(s):=x^{i*}(s,u^{i*})$ as the optimal opinion which is the solution of the Equation (\ref{1}) with the initial opinion $x_0^i$. First objective is to determine the G\^ateaux derivative of the cost functional $L^i(s,\bm x,u^i)$ at $u^i$ in all directions \citep{valdez2025association,valdez2025exploring}. Consider another strategy $v^i$ such that $v^i(s)=u^{i'}(s)-u^i(s)$ for another admissible strategy $u^{i'}\in\mathcal U([0,t])$. Hence, $v^i\in \mathcal U([0,t])$. $v^i$ can be considered as the direction of the G\^ateaux derivative of $L^i(s,\bm x,u^i)$ \citep{carmona2016}. For every $\epsilon>0$ small enough, define a strategy $u^{i\epsilon}(s)=u^i(s)+\epsilon v^i(s)$, and the corresponding controlled opinion vector $\bm x^{\epsilon}:=\bm x^{\epsilon}(s,u^i)$. Furthermore, define the variational process ${\mathfrak V}=\{\mathcal V^i(s)\}_{s=0}^t$ as the solution of the equation
\begin{multline}
d\mathcal V^i(s)=\left[\frac{\partial}{\partial x^i}\mu^i(s,x^i(s),\mathbb P_{(x^i)},u^i(s))\mathcal V^i(s)+\zeta\left(s,\mathbb P_{(x^i,\mathcal V^i)}\right)+\frac{\partial}{\partial u^i}\mu^i(s,x^i(s),\mathbb P_{(x^i)},u^i(s))\right]ds\\
+\left[\frac{\partial}{\partial x^i}\sigma^i(s,x^i(s),\mathbb P_{(x^i)},u^i(s))\mathcal V^i(s)+\hat\zeta\left(s,\mathbb P_{(x^i,\mathcal V^i)}\right)+\frac{\partial}{\partial u^i}\sigma^i(s,x^i(s),\mathbb P_{(x^i)},u^i(s))\right]dB^i(s),
\end{multline}
where  
\[
\zeta(.)=\tilde\E\left\{\frac{\partial}{\partial \gamma}\mu^i(s,x^i(s),\mathbb P_{(x^i)},u^i(s))(\tilde x^i(s)).\hat {\mathcal V}^i(s)\bigg |_{x^i=x^i(s),u^i=u^i(s)}\right\},
\]
and
\[
\hat\zeta(.)=\tilde\E\left\{\frac{\partial}{\partial \gamma}\sigma^i(s,x^i(s),\mathbb P_{(x^i)},u^i(s))(\tilde x^i(s)).\hat {\mathcal V}^i(s)\bigg |_{x^i=x^i(s),u^i=u^i(s)}\right\}
\]
with $\{\tilde x^i(s),\hat{\mathcal V}^i(s)\}$ being independent copy of $\{x^i(s),\mathcal V^i(s)\}$. A Fr\'echet differentiability has been used to define $\tilde\E$. This type of functional analytic differentiability was introduced by Pierre Lions at the \emph{Coll\'ege de France} \citep{carmona2015,carmona2016}. This is a type of differentiability based on the lifting of functions $\mathcal P_2(\mathbb R^n)\ni\gamma\hookrightarrow \mathcal H(\gamma)$ into functions $\hat{\mathcal H}$ defined on Hilbert space $\mathcal H^{2,\tilde m}(\tilde\Omega;\mathbb R^n)$ on some probability space $(\tilde\Omega,\tilde{\mathcal F},\tilde{\mathbb P})$ after setting $\hat{\mathcal H}(\tilde{\bm x})=\mathcal H(\tilde{\mathbb P}_{\tilde{\bm x}})$ for all $\tilde{\bm x}\in\mathcal H^{2,\tilde m}(\tilde\Omega;\mathbb R^n)$, with $\tilde\Omega$ being a Polish space and $\tilde{\mathbb P}$ an atomless measure \citep{carmona2015}. Since there are $n$ number of agents in the system such that $n\ra\infty$, instead of considering the opinions of the other agents, agent $i$ considers the distribution of all opinions in the system $\mathcal H(\tilde{\mathbb P}_{\tilde{\bm x}})$ and makes their opinions \citep{khan2023myb0}. Therefore, in this case the distribution function of opinions $\mathcal H$ is said to be differentiable at $\bar\gamma\in\mathcal P_2(\mathbb R^n)$ if there exists a set of random opinions $\tilde{\bm x}_*$ with probability distribution $\bar\gamma$ (i.e., $\tilde{\mathbb P}_{\tilde{\bm x}_*}=\bar\gamma$). The Fr\'echet derivative of $\hat{\mathcal H}$ at $\tilde{\bm x}_*$ is the element of Hilbert space $\mathcal H^{2,\tilde m}(\tilde\Omega;\mathbb R^n)$ by identifying itself and its dual \citep{carmona2015}. One important of the Fr\'echet differentiation in this type of environment is that the distribution of derivative depends on $\bar \gamma$, not on $\tilde{\bm x}_*$. Fr\'echet derivative of $\mathcal H$ is
\[
\mathcal H(\gamma)=\mathcal H(\hat\gamma)+[D_f\hat{\mathcal H}](\tilde{\bm x}_*).(\tilde{\bm x}-\tilde{\bm x}_*)+o\left(\| \tilde{\bm x}-\tilde{\bm x}_*\|_2\right),
\]
where $[D_f\hat{\mathcal H}](\tilde{\bm x}_*)$ is the Fr\'echet derivative, the dot is the inner product of the Hilbert space over $(\tilde\Omega,\tilde{\mathcal F},\tilde{\mathbb P})$, and $\|.\|_2$ is a norm of that Hilbert space. For a deterministic function $\tilde g:\mathbb R^n\hookrightarrow \mathbb R^n$ it is well understood that Fr\'echet derivative of the form $\tilde g(\tilde{\bm x}_*)$ is uniquely defined $\hat\gamma$ almost everywhere in $\mathbb R$ \citep{cardaliaguet2012,carmona2016}. The unique equivalence class of $\tilde g$ is denoted by $\partial_\gamma\mathcal H(\bar\gamma)$. $\partial_\gamma\mathcal H(\bar\gamma)$ is the partial derivative of $\mathcal H$ at $\bar\gamma$ such that 
\[
\partial_\gamma\mathcal H(\bar\gamma)(.):\mathbb R^n\hookrightarrow \partial_\gamma\mathcal H(\bar\gamma)(y)\in\mathbb R^n.
\]
The partial derivative $\partial_\gamma\mathcal H(\bar\gamma)$ allows to express the Fr\'echet derivative $[D_f\hat{\mathcal H}](\tilde{\bm x}_*)$ as a function of any random variable $\tilde{\bm x}_*$ with its law $\bar\gamma$ irrespective of the definition of $\tilde{\bm x}_*$. If $\mathcal H(\gamma)=\int_{\mathbb R^n}g(y)\gamma(dy)=\langle g,\gamma\rangle$ for some scalar differentiable function $g$ on $\mathbb R^n$. Here $\hat{\mathcal H}(\tilde{\bm y})=\tilde\E[g(\tilde{\bm y})]$ and $D_f\hat{\mathcal H}(\tilde{\bm y}).(\tilde{\bm x})=\tilde\E[\partial g(\tilde{\bm y}).(\tilde{\bm x})]$, where $\partial_\gamma\mathcal H(\gamma)$ is thought to be a deterministic function $\partial g$ \citep{carmona2015}.

\begin{lem}\label{lem2}
	For a small $\epsilon>0$, the admissible strategy $u^{i\epsilon}$ defined as $u^{i\epsilon}(s)=u^i(s)+\epsilon v^i(s)$, with opinion of agent $i$ as $\bm x^{\epsilon}:=\bm x^{\epsilon}(s,u^i)$ following condition holds
	\[
	\lim_{\epsilon\ra 0}\E\left\{\sup_{s\in[0,t]}\left|\mathcal V^i(s)-\frac{x^{i\epsilon}(s)-x^i(s)}{\epsilon}\right|\right\}=0,
	\]
	where $x^{i\epsilon}(s)\in\bm x^\epsilon$ is agent $i$'s opinion at time s coming from the set of all opinions in the environment.	
\end{lem}

\begin{proof}
	See the Appendix.
\end{proof}	

\begin{remark}
	Based on the Assumptions \ref{as0}-\ref{as2} Lemma \ref{lem2} guarantees the existance and the uniqueness of $\mathcal V^i(s)$. Furthermore, for any $\varrho\in[1,\infty)$ this $\mathcal V^i(s)$ satisfies $\E\left\{\sup_{s\in[0,t]}|\mathcal V^i(s)|^\varrho\right\}<\infty$. Above Lemma \ref{lem2} in some Hilbert space $\mathcal H^{2,\tilde m}$, $\mathcal V^i(s)$ is derivative of the opinion driven by $i^{th}$ agent's strategy when the direction of the derivative $v^i(s)$ is changed.
\end{remark}	

\begin{lem}\label{lem3}
	For $\epsilon>0$ small enough and the time interval $[s,s+\epsilon]\subset[0,t]$ there exists some $\delta\in[0,\epsilon)$ so that the admissible strategy function of agent $i$ denoted as $u^i(s)\hookrightarrow L^i(s,\bm x,u^i)$ is G\^ateaux differentiable and
	\[
	\frac{\partial}{\partial \delta}L^i(s,\bm x,u^i)\bigg|_{\delta=0}=\E_s\left\{\int_s^{s+\epsilon}\left[\mathcal V^i(\nu)\sum_{j\in\eta_i}\left(w_{ij}\left[x^i(\nu)-x^j(\nu)\right]+k_i\left[x^i(\nu)-x_0^i\right]\right)+u^i(\nu)v^i(\nu)\right]d\nu\right\},
	\]
	where $\E_s\{.\}=\E\{.|x^i(s)\}$ for all $\nu\in[s,s+\epsilon]$.
\end{lem}	
\begin{proof}
	See the Appendix.
\end{proof}

\begin{remark}
	Above Lemma determines the directional derivative of the cost functional $L^i(s,\bm x,u^i)$ for some $\epsilon>0$ small enough, $[s,s+\epsilon]\subset[0,t]$.
\end{remark}

\section{The Adjoint Processes.}
Let for $s\in[0,t]$, $g(s):[p,q]\ra\mathcal{C}$ be an opinion dynamics of $i^{th}$ agent with initial and terminal points $g(p)$ and $g(q)$ respectively, such that, the line path integral is $\int_{\mathcal{C}} f(\gamma) ds=\int_{p}^q f(g(s))|g'(s)| ds$,  where $g'(s)=\partial g(s)/\partial s$. In this paper a functional path integral approach is considered where the domain of the integral is assumed to be the space of functions \citep{pramanik2020}. In \cite{feynman1948} theoretical physicist Richard Feynman introduced \emph{Feynman path integral}, and popularized it in quantum mechanics. Furthermore, mathematicians develop the measurability of this functional integral and in recent years it has become popular in probability theory \citep{fujiwara2017}. In quantum mechanics, when a particle moves from one point to another, between those points it chooses the shortest path out of infinitely many paths such that some of them touch the edge of the universe. After introducing $n$ number of equal lengthed small intervals $[s,s+\epsilon]\subset[0,t]$ with $\epsilon>0$ small enough, and using the \emph{Riemann–Lebesgue} lemma if at time $s$ one particle touches the end of the universe, then at a later time point, it would come back and go to the opposite side of the previous direction to make the path integral a measurable function \citep{bochner1949}. Similarly, since agent $i$ has infinitely many opinions, they choose the opinion associated with least cost given by the constraint explained in Equation (\ref{1}). Furthermore, the Feynman approach is useful in both linear and non-linear stochastic differential equation systems where constructing an HJB equation numerically is quite difficult \citep{belal2007}. 

\begin{definition}\label{d0}
	For a particle, let $\hat{\mathcal{L}}[s,y(s),\dot y(s)]=(1/2) \hat m\dot y(s)^2-\hat{V}(y)$ be the Lagrangian in classical sense in generalized coordinate $y$ with mass $\hat m$ where $(1/2)\hat m\dot y^2$ and $\hat V(y)$ are kinetic and potential energies respectively. The transition function of Feynman path integral corresponding to the classical action function $Z^*=\int_0^t \hat{\mathcal{L}}(s,y(s),\dot y(s)) ds$  is defined as $\Psi(y)=\int_{\mathbb R} \exp\{{Z^*}\} \mathcal{D}_Y $, where $\dot y=\partial y/\partial s$ and $\mathcal{D}_Y$ is an approximated Riemann measure which represents the positions of the particle at different time points $s$ in $[0,t]$ \citep{pramanik2020}.
\end{definition}

\begin{remark}
	Definition \ref{d0} describes the construction of the Feynman path integral in physical sense. This definition is important to construct the stochastic Lagrangian of agent $i$.
\end{remark}	

From Equation (45) of \cite{ewald2024adaptation} for agent $i$,  the stochastic Lagrangian at time $s\in[0,t]$ is defined as
\begin{multline}\label{3}
\hat{\mathcal{L}}^i\left(s,\bm x,\mathbb P_{(\bm x)},\lambda^i,u^i\right)=\E\biggr\{\mbox{$\frac{1}{2}$}\int_0^t \bigg\{\sum_{j\in\eta_i}w_{ij}\left[x^i(s)-x^j(s)\right]^2+k_i\left[x^i(s)-x_0^i\right]^2+\left[u^i(s)\right]^2\bigg\}ds\\
+\int_0^t\left[x^i(s)-x_0^i-\int_0^s[\mu^i[\nu,x^i(\nu),\mathbb P_{(x^i)},u^i(\nu)]d\nu-\sigma^i[\nu,x^i(\nu),\mathbb P_{(x^i)},u^i(\nu)]dB^i(\nu)]\right] d\lambda^i(s)\biggr\},
\end{multline}
where $\lambda^i(s)$ is the Lagrangian multiplier. 

\begin{prop}\label{p0.0}
	\citep{love1993note,ewald2024adaptation}. Suppose for agent $i$, $u^{i*}(s)$ is an admissible strategy and $x^{i*}(s)$ is the corresponding opinion. Furthermore, assume that there exists a progressively measurable Lagrangian multiplier $\lambda^{i*}(s)$ so that following two conditions hold,
	\begin{align}
	\frac{\partial \mathcal L^i}{\partial x^i}\left[s,\bm x^*(s),\mathbb P_{(\bm x^*)},\lambda^{i*}(s),u^{i*}(s)\right]&=0,\label{5.0}\\
	\frac{\partial \mathcal L^i}{\partial u^i}\left[s,\bm x^*(s),\mathbb P_{(\bm x^*)},\lambda^{i*}(s),u^{i*}(s)\right]&=0.\label{5.1}
	\end{align}
	Moreover, assume that the mapping
	\begin{equation}\label{5.2}
	(x^i,u^i)\mapsto L^i(s,\bm x,u^i)+\mu^i[\nu,x^i,\mathbb P_{(x^i)},u^i]\lambda^{i*}(s)+\sigma^i[\nu,x^i,\mathbb P_{(x^i)},u^i]\frac{d\lambda^{i*}(s)dB(s)}{ds},
	\end{equation}
	is concave in $s\in[0,t]$ almost surely. Therefore, the admissible strategy $u^{i*}(s)$ is an optimal strategy of agent $i$.
\end{prop}	
\begin{remark}
	Following \cite{ewald2024adaptation} we know that, if $u^{i*}$ is the solution of the system represented by Equations (\ref{0}), (\ref{1}) and Condition \ref{5.2}, then there exists a progressively measurable It\^o process $\lambda^{i*}(s)$ such that Equations \ref{5.0} and \ref{5.1} hold. In Proposition \ref{p0.0} the Lagrangian multiplier $\lambda^{i*}(s)$ is indeed a progressively measurable process.
\end{remark}

Since at the beginning of the continuous interval $[s,s+\epsilon]$ for all $\epsilon\downarrow 0$, agent $i$ does not have any future information to build their opinion. Thus, $\E_{[s,s+\epsilon]}\{.\}\simeq\E_s\{.\}=\E\{.|x^i(s)\}$. Furthermore, as $\epsilon\downarrow 0$, the Lagrangian expressed in Equation \eqref{3} becomes
\begin{align}\label{5.3}
&\mathcal L^i\left(s,\bm x,\mathbb P_{(\bm x)},\lambda^i,u^i\right)\notag\\
&:=\lim_{\epsilon\downarrow 0}\E_s\biggr\{\mbox{$\frac{1}{2}$}\int_s^{s+\epsilon} \bigg\{\sum_{j\in\eta_i}w_{ij}\left[x^i(\nu)-x^j(\nu)\right]^2+k_i\left[x^i(\nu)-x_0^i\right]^2+\left[u^i(\nu)\right]^2\bigg\}d\nu\notag\\
&\hspace{.1cm}+\int_s^{s+\epsilon}\left[x^i(\nu)-x_0^i-\int_s^{\nu}[\mu^i[\hat\nu,x^i(\hat\nu),\mathbb P_{(x^i)},u^i(\hat\nu)]d\hat\nu-\sigma^i[\hat\nu,x^i(\hat\nu),\mathbb P_{(x^i)},u^i(\hat\nu)]dB^i(\hat\nu)]\right] d\lambda^i(\nu)\biggr\}\notag\\
&\approx\E_s\biggr\{\mbox{$\frac{1}{2}$}\bigg\{\sum_{j\in\eta_i}w_{ij}\left[x^i(s)-x^j(s)\right]^2+k_i\left[x^i(s)-x_0^i\right]^2+\left[u^i(s)\right]^2\bigg\}ds\notag\\
&\hspace{1cm}+\left[x^i(s)-x_0^i-\mu^i[s,x^i(s),\mathbb P_{(x^i)},u^i(s)]-\sigma^i[s,x^i(s),\mathbb P_{(x^i)},u^i(s)]\right]d\lambda^i(s)\bigg\},
\end{align}
where $[s,\nu]\subset[s,s+\epsilon]$. 

The adjoint process of the system is
\begin{multline}\label{5}
d\lambda_1^i(s)=-\biggr[\frac{\partial}{\partial x}\mu^i[s,x^i(s),\mathbb P_{(x^i)},u^i(s)]\lambda_1^i(s)+\frac{\partial}{\partial x}\sigma^i[s,x^i(s),\mathbb P_{(x^i)},u^i(s)]\lambda_1^i(s)\\+\frac{\partial}{\partial x}L^i(s,\bm x,u^i)\biggr]ds+\lambda_2^i(s)dB^i(s),
\end{multline}
where $\lambda_1^i(s)$ and $\lambda_2^i(s)$ are two new dual variables belong to the dual spaces of the spaces from where $\mu^i$ and $\sigma^i$ take their values such that $\lambda_1^i\in\mathbb R$ like $x^i$, and that $\lambda_2^i\in\mathbb R^2$. Notice that the deterministic Hamiltonian of the system is
\[
H^i\left(s,\bm x,\mathbb P_{(\bm x)},\lambda_1^i,\lambda_2^i,u^i\right)=L^i(s,\bm x,u^i)+\lambda_1^i(s)\mu^i[s,x^i(s),\mathbb P_{(x^i)},u^i(s)]+\lambda_2^i(s)\sigma^i[s,x^i(s),\mathbb P_{(x^i)},u^i(s)].
\]
The differences between the above Hamiltonian and the Equation (\ref{3}) are the presence of $\Delta x^i(s):=x^i(s)-x^i_0$, $\lambda^i(s)$, $\E_s\{.\}$, $ds$ and $dB^i(s)$. If $\Delta x^i(s)\ra 0$, under deterministic case Hamiltonian and Lagrangian share a similar structure. Since the Feynman path integral approach has been used, $dB^i(s)$ determines true fluctuation of $\mathcal L^i$ and further inclusion of $\E_s$ facilitates the conditional expectation of a forward looking process for $[s,s+\epsilon]$. The Lagrangian used in Equation (\ref{3}) is stochastic but the usual Hamiltonian of the control theory is deterministic.

\begin{definition}\label{d1}
	For a set of admissible strategies $\bm u^i=\{u^i(s)\}_{s=0}^t\in\mathcal U([0,t])$ of agent $i$, denote $\bm x^{i}(s)=\bm x^{i}(s,u^{i})$ the set of corresponding controlled opinions, and let $(\bm \lambda_1^i,\bm \lambda_2^i)=\{ \lambda_1^i(s),\lambda_2^i(s)\}_{s=0}^t$ be any coupled adjoint progressively measurable stochastics processes satisfying
	\[
	d\lambda_1^i(s)=-\partial_x	H^i(s,\bm x,\mathbb P_{\bm x},\lambda^i_1,\lambda^i_2,u^i)+\lambda^i_2(s)dB^i(s)-\tilde\E\left[\partial_\gamma\tilde{H}^i(s,\tilde{\bm x},\mathbb P_{\tilde{\bm x}},\tilde\lambda^i_1,\tilde\lambda^i_2,\tilde u^i)\right][x^i(s)],
	\]
	where $(\tilde{\bm x},\tilde\lambda^i_1,\tilde\lambda^i_2,\tilde u^i,\tilde{\mathcal L}^i)$ is the independent copy of $(\bm x,\lambda^i_1,\lambda^i_2,u^i,\mathcal L^i)$, and $\tilde\E$ is the expectation of the independent copy. In the adjoint equation $\partial_x=\partial/\partial x$ and $\partial_\gamma=\partial/\partial \gamma$.
\end{definition}

\begin{remark}
	If $\mu^i$ and $\sigma^i$ are independent with the marginal distributions of the process, the extra terms appearing in the adjoint equation in Definition \ref{d1} vanishes and indeed this equation becomes the classical adjoint equation. 
\end{remark}
In the present set up the adjoint equation can be written as
\begin{multline*}
d\lambda_1^i(s)=-\biggr[\partial_x\mu^i[s,x^i(s),\mathbb P_{(x^i)},u^i(s)]\lambda_1^i(s)\\+\partial_x\sigma^i[s,x^i(s),\mathbb P_{(x^i)},u^i(s)]\lambda_1^i(s)\biggr]ds+\partial_xL^i(s,\bm x,u^i)+\lambda_2^i(s)dB^i(s)\\
-\tilde\E\left\{\left[\partial_\gamma\tilde{\mu}^i((s,\tilde{\bm x},\bm x,\tilde\lambda^i_1,\tilde\lambda^i_2,\tilde u^i))+\partial_\gamma\tilde{\sigma}^i(s,\tilde{\bm x},\bm x,\tilde\lambda^i_1,\tilde\lambda^i_2,\tilde u^i)\right]ds+\partial_\gamma\tilde{L}^i(s,\tilde{\bm x},\bm x,\tilde\lambda^i_1,\tilde\lambda^i_2,\tilde u^i)\bigg|_{\bm x=\bm x(s)}\right\}.
\end{multline*}	

It is important to note that for a given admissible strategy $u^i\in\mathcal U([0,t])$ and the controlled opinion $x^i$, despite the boundedness assumptions of the partial derivatives of $\mu^i$ and $\sigma^i$, and despite that the first part of the above adjoint equation being linear with respect to $\lambda_1^i(s)$ and $\lambda_2^i(s)$, existence and uniqueness of a solution $\{\bm\lambda_1^{i*},\bm\lambda_2^{i*}\}$ of the adjoint equation can not be determined by standard process (for example Theorem 2.2 in \cite{carmona2016}). The main reason is the joint distribution of solution  process appears in $\mu^i$ and $\sigma^i$ \citep{carmona2015,carmona2016}.

\begin{lem}\label{lem4}
	Under (v) of Assumption \ref{as1} there exists a unique adapted solution $(\bm\lambda_1^{i*},\bm\lambda_2^{i*})$ of the coupled adjoint progressively measurable stochastics processes satisfying
	\[
	d\lambda_1^i(s)=-\partial_x	\mathcal L^i(s,\bm x,\mathbb P_{\bm x},\lambda^i_1,\lambda^i_2,u^i)+\lambda^i_2(s)dB^i(s)-\tilde\E\left[\partial_\gamma\tilde{\mathcal L}^i(s,\tilde{\bm x},\mathbb P_{\tilde{\bm x}},\tilde\lambda^i_1,\tilde\lambda^i_2,\tilde u^i)\right][x^i(s)],
	\]
	in 	$\mathcal H_{\bm\lambda_1}^{2,\tilde m}\bigotimes\mathcal H_{\bm\lambda_2}^{2,\tilde m}$, where
	\[
	\mathcal H_{\bm\lambda_1}^{2,\tilde m}:=\left\{\bm\lambda_1^i\in\mathcal H_{\bm\lambda_1}^{0,\tilde m};\ \E\int_0^t|\lambda_1^i(s)|^2ds<\infty\right\},
	\]
	and
	\[
	\mathcal H_{\bm\lambda_2}^{2,\tilde m}:=\left\{\bm\lambda_2^i\in\mathcal H_{\bm\lambda_1}^{0,\tilde m};\ \E\int_0^t|\lambda_2^i(s)|^2ds<\infty\right\}.
	\]
\end{lem}
\begin{proof}
	See the Appendix.
\end{proof}

\begin{remark}
	Lemma \ref{lem4} states that for each admissible strategy $u^i$, there exists a couple of adjoint processes $\left(\bm\lambda_1^i,\bm\lambda_2^i\right)$ so that $\E\left\{\sup_{s\in[0,t]}\left|\lambda_1^i(s)\right|^2\right\}+\E\left\{\int_0^t\left|\lambda_2^i(s)\right|^2ds\right\} $.	
\end{remark}	

For a normalizing constant $L_\epsilon^i>0$  define a transition function from $s$ to $s+\epsilon$ as
\begin{equation}\label{lin8}
\Psi_{s,s+\varepsilon}^i(x^i):=\frac{1}{L_\varepsilon^i} \int_{\mathbb{R}^{n}}\exp[-\varepsilon \mathcal{A}_{s,s+\varepsilon}(x^i)]\Psi_s^i(x^i)dx^i(s),
\end{equation}
where $\Psi_s^i(x^i)$ is the value of the transition function based on opinion $x^i$ at time $s$ with the initial condition $\Psi_0^i(x^i)=\Psi_0^i$. The penalization constant $L_\epsilon^i$ is chosen in such a way that the right hand side of the expression \ref{lin8} becomes unity. Therefore, the action function of agent $i$ is,
\begin{multline*}
\mathcal{A}_{s,s+\varepsilon}(x^i)=
\int_{s}^{s+\varepsilon}\E_\nu\biggr\{\mbox{$\frac{1}{2}$} \bigg\{\sum_{j\in\eta_i}w_{ij}\left[x^i(\nu)-x^j(\nu)\right]^2+k_i\left[x^i(\nu)-x_0^i\right]^2+\left[u^i(\nu)\right]^2\bigg\}d\nu\\+h^i[\nu+\Delta \nu,x^i(\nu)+\Delta x^i(\nu)]d\lambda^i(\nu)\biggr\},
\end{multline*}
where $h^i[\nu+\Delta \nu,x^i(\nu)+\Delta x^i(\nu)]\in C^2([0,t]\times\mathbb{R})$ is an It\^o process so that,
\begin{equation*}
h^i[\nu+\Delta \nu,x^i(\nu)+\Delta x^i(\nu)]\approx x^i(\nu)-x_0^i-\mu^i[\nu,x^i(\nu),\mathbb P_{(x^i)},u^i(\nu)]d\nu-\sigma^i[\nu,x^i(\nu),\mathbb P_{(x^i)},u^i(\nu)]dB^i(\nu).
\end{equation*}
The action $\mathcal{A}_{s,s+\epsilon}(x^i)$ tells us within $[s,s+\epsilon]$ the action of agent $i$ depends on their opinion  $x^i$ under a feedback structure.

\begin{definition}\label{d2}
	For optimal opinion  $x^{i*}(s)$ and there exists an optimal admissible control $u^{i*}(s)$ such that for all $s\in[0,t]$ the conditional expectation of the cost function is
	\begin{multline*}
	\E_0 \left[\int_{0}^t\mbox{$\frac{1}{2}$} \bigg\{\sum_{j\in\eta_i}w_{ij}\left[x^{i*}(s)-x^{j*}(s)\right]^2+k_i\left[x^{i*}(s)-x_0^i\right]^2+\left[u^{i*}(s)\right]^2\bigg\}ds\bigg|\mathcal F_0^{x^*} \right]
	\\\geq\E_0 \left[\int_{0}^t\mbox{$\frac{1}{2}$} \bigg\{\sum_{j\in\eta_i}w_{ij}\left[x^i(s)-x^j(s)\right]^2+k_i\left[x^i(s)-x_0^i\right]^2+\left[u^i(s)\right]^2\bigg\}ds\bigg|\mathcal F_0^{x}\right],
	\end{multline*}
	such that Equation (\ref{1}) holds, where $\mathcal F_0^{x^*}$ is the optimal filtration satisfying $\mathcal F_0^{x^*}\subset\mathcal F_0^{x}$.
\end{definition}

\section{Main results.}

Consider for an opinion space $X_0=\{\mathbf x(s):s\in[0,t]\}$, and agent $i$'s control space $\mathcal U$ there exists an admissible control $u^i:[0,t]\times X_0\ra\mathcal U$ and for all $i\in N$ define the integrand of the cost function $L^i(.)$ as
\[
L^i(s,\mathbf x,u^i)=\E\left\{\mbox{$\frac{1}{2}$}\int_0^t \bigg(\sum_{j\in\eta_i}w_{ij}\left[x^i(s)-x^j(s)\right]^2+k_i\left[x^i(s)-x_0^i\right]^2+\left[u^i(s)\right]^2\bigg)ds\right\}.
\]

\begin{prop}\label{p0}
	For agent $i$\\
	(i).  the feedback control $u^i(s,x^i):[0,t]\times \mathbb{R}\times\mathcal U\ra\mathbb{R}\times\mathcal U$ is a continuously differentiable function,\\
	(ii). The cost functional $L^i(s,\mathbf x,u^i):[0,t]\times\mathbb{R}^n\times\mathbb{R}\ra\mathbb{R}$ is smooth on $\mathbb{R}\times\mathcal U$.\\
	(iii). If $X_0=\{\mathbf x(s), s\in[0,t]\}$ is an opinion trajectory of agent $i$ then, the feedback Nash equilibrium $\big\{u^{i*}(s,x^i)\in\mathcal U; i\in N\big\}$ would be the solution of the following equation
	\begin{align}\label{11}
	\mbox{$\frac{\partial}{\partial u^i}$}f^i(s,\mathbf x,u^i) \left[\mbox{$\frac{\partial^2}{\partial (x^i)^2}$}f^i(s,\mathbf x,u^i)\right]^2=2\mbox{$\frac{\partial}{\partial x^i}$}f^i(s,\mathbf x,u^i) \mbox{$\frac{\partial^2}{\partial x^i\partial u^i}$}f^i(s,\mathbf x,u^i),
	\end{align}
	where for an It\^o process $h^i(s,x^i)\in \mathcal [0,t]\times \mathcal P_2(\mathbb R)\times\mathbb R$
	\begin{align}\label{12}
	f^i(s,\mathbf x,\gamma,u^i)&=L^i(s,\mathbf x,u^i)+h^i(s,x^i)d\lambda^i(s)+\left[\mbox{$\frac{\partial h^i(s,x^i)}{\partial s}$}d\lambda^i(s)+\mbox{$\frac{d\lambda^i(s)}{d s}$}h^i(s,x^i)\right]\notag\\
	&\hspace{.25cm}+\mbox{$\frac{\partial h^i(s,x^i)}{\partial x^i}$}\mu^i\left[s,x^i,\mathbb P_{(x^i)},u^i\right]d\lambda^i(s)+\mbox{$\frac{1}{2}$}\left[\sigma^{i}\left[s,x^i,\mathbb P_{(x^i)},u^i\right]\right]^2\mbox{$\frac{\partial^2 h^i(s,x^i)}{\partial (x^i)^2}$}d\lambda^i(s).
	\end{align}
\end{prop}

\begin{proof}
	See the Appendix.
\end{proof}	

\begin{remark}
	The central idea of Proposition \ref{p0} is to choose $h^i$ appropriately. Therefore, one natural candidate should be a function of the integrating factor of the stochastic opinion dynamics represented in Equation \eqref{1}.
\end{remark}

To demonstrate the preceding proposition, we present a detailed example to identify an optimal strategy in McKean-Vlasov SDEs and the corresponding systems of interacting particles. Specifically, we examine a stochastic opinion dynamics model involving six unknown parameters. To construct this example we are going to combine the equations from \cite{sharrock2021parameter}, \cite{sharrock2023online}, and \cite{chen2022deep}. Following \cite{sharrock2023online} consider a one-dimensional stochastic opinion dynamics model, parameterized by $\theta=(\theta_1,\theta_2)^T\in\mathbb R^2$ of the form
\begin{equation}\label{6}
dx^i(s)=-\left[\int_{\mathbb R}\phi_\theta\left(||x^i(s)-x^j(s)||\right)\left[x^i(s)-x^j(s)\right]\gamma d(x^j)\right]ds+\sigma^ix^i(s)dB^i(s),
\end{equation}
where $\sigma^i>0$, $B^i=\{B^i(s)\}_{s\geq 0}$ is standard Brownian motion,  and the interaction kernel $\phi_\theta:\mathbb R_+\ra\mathbb R_+$ which has the form
\begin{align}\label{7}
\phi_\theta(\be)&=
\begin{cases}
\theta_1\exp\left\{-\frac{0.01}{1-(\beta-\theta_2)^2}\right\} & \text{if}\ \beta>0\\
0 & \text{if}\ \beta\leq 0.
\end{cases}
\end{align}
This model is often described in terms of the corresponding system of interacting particles, which is represented by
\begin{equation}\label{8}
dx^i(s)=-\frac{1}{n}\sum_{j=1}^n\phi_\theta\left(||x^i(s)-x^j(s)||\right)\left[x^i(s)-x^j(s)\right]ds+\sigma^ix^i(s)dB^i(s),
\end{equation}
where $\theta_1$ is the scale, and $\theta_2$ is the range parameters. Models of this type appear in a variety of fields, ranging from biology to the social sciences, where $\phi_\theta$ determines how the behavior of one particle (such as an agent's opinions) affects the behavior of other particles (such as the opinions of others). For a comprehensive discussion of these models, see \citep{brugna2015kinetic,chazelle2017well,lu2021learning,sharrock2021parameter}. In deterministic versions of these models, it is well established that over time, particles converge into clusters. The number of clusters depends on both the interaction kernel (i.e., the scope and intensity of interactions between particles) and the initial conditions.

Following \cite{chen2022deep} we use a modified version of Friedkin and Johnsen model \citep{friedkin1990social}
\begin{equation}\label{9}
dx^i(s)=-\a(s)x^i(s) ds-\a(s)F^i(\mathbb A_i)ds+G(s,x^i(s))[u^i(s)]^2+\sigma^ix^i(s)dB^i(s),
\end{equation}
where $F^i:[0,1]\times[0,t]\ra[0,1]$, $\mathbb A_i$ defines the set of neighbors of $i^{th}$ agent such that $\mathbb A_i=\left\{x^j(s)\big| ||x^i(s)-x^j(s)||_2^2\leq r\right\}$, for all $r$ be the radius of neighborhood, and $G:[0,t]\times[0,1]\ra[0,1]\times\mathcal U$ be the actuator dynamics \citep{chen2022deep}. After assuming $$F^i(\mathbb A_i)=\frac{1}{n}\sum_{j=1}^n\phi_\theta\left(||x^i(s)-x^j(s)||\right)\left[x^i(s)-x^j(s)\right]$$ and $G(s,x^i(s))=x^i(s),$ Equation \eqref{9} becomes,
\begin{multline}\label{10}
dx^i(s)=-\a(s)x^i(s) ds-\a(s)\frac{1}{n}\sum_{j=1}^n\phi_\theta\left(||x^i(s)-x^j(s)||\right)\left[x^i(s)-x^j(s)\right]ds\\+x^i(s)[u^i(s)]^2+\sigma^ix^i(s)dB^i(s).
\end{multline}
Our main aim is to minimize Equation \eqref{0} subject to Equation \eqref{10}. The integrating factor of Equation \eqref{10} is $\exp\{-\int_0^s\sigma^idB^i(\nu)\\+\frac{1}{2}\int_0^s(\sigma^i)^2d\nu\}$. Therefore, $h^i(s)$ function is
\begin{align}\label{13}
h^i(s)&=x^i(s)\exp\left\{-\int_0^s\sigma^idB^i(\nu)+\frac{1}{2}\int_0^s(\sigma^i)^2d\nu\right\}\notag\\
&=x^i(s)\exp\left\{-\sigma^i B^i(s)+\frac{1}{2}(\sigma^i)^2 s\right\}.
\end{align}
Equation \eqref{12} becomes,
\begin{multline}\label{14}
f^i(s,\mathbf x,\gamma,u^i)\\
=\frac{1}{2}\bigg\{\sum_{j\in\eta_i}w_{ij}\left[x^i(s)-x^j(s)\right]^2+k_i\left[x^i(s)-x_0^i\right]^2+\left[u^i(s)\right]^2\bigg\}+x^i(s)\exp\left\{-\sigma^i B^i(s)+\frac{1}{2}(\sigma^i)^2 s\right\}d\lambda^i(s)\\
+\frac{1}{2}x^i(s)\exp\left\{-\sigma^i B^i(s)+\frac{1}{2}(\sigma^i)^2 s\right\}\left(\sigma^i\right)^2d\lambda^i(s)+\frac{d\lambda^i(s)}{ds}x^i(s)\exp\left\{-\sigma^i B^i(s)+\frac{1}{2}(\sigma^i)^2 s\right\}\\
+\exp\left\{-\sigma^i B^i(s)+\frac{1}{2}(\sigma^i)^2 s\right\}\biggr[-\a(s)x^i(s)-\a\frac{1}{n}\sum_{j=1}^n\phi_\theta\left(||x^i(s)-x^j(s)||\right)\\\times\left[x^i(s)-x^j(s)\right]+x^i(s)(u^i(s))^2\biggr]d\lambda^i(s).
\end{multline}
Now,
\begin{align*}
&\frac{\partial}{\partial x^i}f^i(s,\mathbf x,\gamma,u^i)\\
&\hspace{.5cm}=\sum_{j\in\eta_i}w_{ij}\left[x^i(s)-x^j(s)\right]+k_i\left[x^i(s)-x_0^i\right]+\exp\left\{-\sigma^i B^i(s)+\frac{1}{2}(\sigma^i)^2 s\right\}d\lambda^i(s)\\
&\hspace{1cm}+\exp\left\{-\sigma^i B^i(s)+\frac{1}{2}(\sigma^i)^2 s\right\}\frac{d\lambda^i(s)}{ds}+\exp\left\{-\sigma^i B^i(s)+\frac{1}{2}(\sigma^i)^2 s\right\}\\
&\hspace{1.5cm}\times\biggr\{-\a(s)-\a(s)\frac{1}{n}\sum_{j=1}^n\biggr[\frac{d}{dx^i}\phi_\theta\left(||x^i(s)-x^j(s)||\right)\left[x^i(s)-x^j(s)\right]\\
&\hspace{2cm}+\phi_\theta\left(||x^i(s)-x^j(s)||\right)\biggr]+(u^i(s))^2\biggr\}d\lambda^i(s),\\
&\frac{\partial^2}{\partial (x^i)^2}f^i(s,\mathbf x,\gamma,u^i)=\sum_{j\in\eta_i}w_{ij}+k_i+\exp\left\{-\sigma^i B^i(s)+\frac{1}{2}(\sigma^i)^2 s\right\}\\
&\hspace{.5cm}\times\left\{-\a(s)\frac{1}{n}\sum_{j=1}^n\left[\frac{d^2}{d(x^i)^2}\phi_\theta\left(||x^i(s)-x^j(s)||\right)\left[x^i(s)-x^j(s)\right]+2\frac{d}{dx^i}\phi_\theta\left(||x^i(s)-x^j(s)||\right)\right]\right\}d\lambda^i(s),\\
&\frac{\partial}{\partial u^i}f^i(s,\mathbf x,\gamma,u^i)=u^i(s)\left[1+2x^i(s)\exp\left\{-\sigma^i B^i(s)+\frac{1}{2}(\sigma^i)^2 s\right\} d\lambda^i(s)\right],\\
&\frac{\partial^2}{\partial x^i\partial u^i}f^i(s,\mathbf x,\gamma,u^i)=2\exp\left\{-\sigma^i B^i(s)+\frac{1}{2}(\sigma^i)^2 s\right\}d\lambda^i(s).
\end{align*}
Using the above results Equation \eqref{11} yields,
\begin{multline}\label{15}
4\left[u^i(s)\right]^2\exp\left\{-2\sigma^i B^i(s)+(\sigma^i)^2 s\right\}\left[d\lambda^i(s)\right]^2-u^i(s)\left[1+2x^i(s)\exp\left\{-\sigma^i B^i(s)+\frac{1}{2}(\sigma^i)^2 s\right\} d\lambda^i(s)\right]\\
\times\biggr[\sum_{j\in\eta_i}w_{ij}+k_i+\exp\left\{-\sigma^i B^i(s)+\frac{1}{2}(\sigma^i)^2 s\right\}\left\{-\a(s)\frac{1}{n}\sum_{j=1}^n\left[\frac{d^2}{d(x^i)^2}\phi_\theta\left(||x^i(s)-x^j(s)||\right)\left[x^i(s)-x^j(s)\right]\right.\right.\\
\left.\left.+2\frac{d}{dx^i}\phi_\theta\left(||x^i(s)-x^j(s)||\right)\right]\right\}d\lambda^i(s)\biggr]\\
+4\exp\left\{-\sigma^i B^i(s)+\frac{1}{2}(\sigma^i)^2 s\right\}d\lambda^i(s)\biggr[\sum_{j\in\eta_i}w_{ij}\left[x^i(s)-x^j(s)\right]+k_i\left[x^i(s)-x_0^i\right]\\
+\exp\left\{-\sigma^i B^i(s)+\frac{1}{2}(\sigma^i)^2 s\right\}d\lambda^i(s)+\exp\left\{-\sigma^i B^i(s)+\frac{1}{2}(\sigma^i)^2 s\right\}\frac{d\lambda^i(s)}{ds}+\exp\left\{-\sigma^i B^i(s)+\frac{1}{2}(\sigma^i)^2 s\right\}\\
\times\left\{-\a(s)-\a(s)\frac{1}{n}\sum_{j=1}^n\left[\frac{d}{dx^i}\phi_\theta\left(||x^i(s)-x^j(s)||\right)\left[x^i(s)-x^j(s)\right]+\phi_\theta\left(||x^i(s)-x^j(s)||\right)\right]\right\}d\lambda^i(s)\biggr]=0.
\end{multline}
Clearly, Equation \eqref{15} is a quadratic equation with respect to strategy $u^i(s)$ and can be written as $T_1\left[u^i(s)\right]^2+T_2u^i(s)+T_3=0$. Therefore, optimal strategy of agent $i$ is 
\begin{equation}\label{16}
u^{i*}(s)=\frac{-T_2\pm\sqrt{\left(T_2\right)^2-4T_1T_3}}{2T_1},
\end{equation}
where
\begin{align*}
T_1&=4\exp\left\{-2\sigma^i B^i(s)+(\sigma^i)^2 s\right\}\left[d\lambda^i(s)\right]^2> 0,\\
T_2&=-\left[1+2x^i(s)\exp\left\{-\sigma^i B^i(s)+\frac{1}{2}(\sigma^i)^2 s\right\} d\lambda^i(s)\right]\biggr[\sum_{j\in\eta_i}w_{ij}+k_i+\exp\left\{-\sigma^i B^i(s)+\frac{1}{2}(\sigma^i)^2 s\right\}\\
&\hspace{.5cm}\times\left\{-\a(s)\frac{1}{n}\sum_{j=1}^n\left[\frac{d^2}{d(x^i)^2}\phi_\theta\left(||x^i(s)-x^j(s)||\right)\left[x^i(s)-x^j(s)\right]+2\frac{d}{dx^i}\phi_\theta\left(||x^i(s)-x^j(s)||\right)\right]\right\}d\lambda^i(s)\biggr],\\
T_3&=4\exp\left\{-\sigma^i B^i(s)+\frac{1}{2}(\sigma^i)^2 s\right\}d\lambda^i(s)\biggr[\sum_{j\in\eta_i}w_{ij}\left[x^i(s)-x^j(s)\right]+k_i\left[x^i(s)-x_0^i\right]\\
&+\exp\left\{-\sigma^i B^i(s)+\frac{1}{2}(\sigma^i)^2 s\right\}d\lambda^i(s)+\exp\left\{-\sigma^i B^i(s)+\frac{1}{2}(\sigma^i)^2 s\right\}\frac{d\lambda^i(s)}{ds}+\exp\left\{-\sigma^i B^i(s)+\frac{1}{2}(\sigma^i)^2 s\right\}\\
&\times\left\{-\a(s)-\a(s)\frac{1}{n}\sum_{j=1}^n\left[\frac{d}{dx^i}\phi_\theta\left(||x^i(s)-x^j(s)||\right)\left[x^i(s)-x^j(s)\right]+\phi_\theta\left(||x^i(s)-x^j(s)||\right)\right]\right\}d\lambda^i(s)\biggr],
\end{align*}
and 
\begin{align*}
&\frac{d}{dx^i}\phi_\theta\left(||x^i(s)-x^j(s)||\right)\\
&=-0.02\theta_1\exp\left\{-0.01\left[1-\left[x^i(s)-x^j(s)-\theta_2\right]^2\right]^{-1}\right\}\\
&\hspace{2cm}\times\left[1-\left[x^i(s)-x^j(s)-\theta_2\right]^2\right]^{-2}\left[x^i(s)-x^j(s)-\theta_2\right],\\
&\frac{d^2}{d(x^i)^2}\phi_\theta\left(||x^i(s)-x^j(s)||\right)\\
&=-0.02\theta_1\biggr[-0.02\exp\left\{-0.01\left[1-\left[x^i(s)-x^j(s)-\theta_2\right]^2\right]^{-1}\right\}\biggr]\\
&\hspace{1cm}+\exp\left\{-0.01\left[1-\left[x^i(s)-x^j(s)-\theta_2\right]^2\right]^{-1}\right\}\\
&\hspace{1.5cm}\times\left[4\left[1-\left[x^i(s)-x^j(s)-\theta_2\right]^2\right]^{-3}\left[x^i(s)-x^j(s)-\theta_2\right]^2+\left[1-\left[x^i(s)-x^j(s)-\theta_2\right]^2\right]^{-2}\right].
\end{align*}
Let us discuss how the optimal strategy of agent $i$ derived in Equation \eqref{16} varies with the difference in the opinions between $i^{th}$ and $j^{th}$ agents.

{\bf Case I.}\\
Suppose there is no difference in opinions between agents $i$ and $j$. In other words, $x^i(s)-x^j(s)=0$, which implies $\phi_\theta\left(||x^i(s)-x^j(s)||\right)=0$, and 
\begin{align*}
T_1&=4\exp\left\{-2\sigma^i B^i(s)+(\sigma^i)^2 s\right\}\left[d\lambda^i(s)\right]^2\neq 0,\\
T_2&=-\left(\sum_{j\in\eta_i}w_{ij}+k_i\right)\left[1+2x^i(s)\exp\left\{-\sigma^i B^i(s)+\frac{1}{2}(\sigma^i)^2 s\right\} d\lambda^i(s)\right],\\
T_3&=4\exp\left\{-2\sigma^i B^i(s)+(\sigma^i)^2 s\right\}d\lambda^i(s)\biggr[k_i\left[x^i(s)-x^i_0\right]\exp\left\{\sigma^i B^i(s)-\frac{1}{2}(\sigma^i)^2 s\right\}\\
&\hspace{2cm}+d\lambda^i(s)+\frac{d\lambda^i(s)}{ds}-\a(s)d\lambda^i\biggr].
\end{align*}

Therefore,
\begin{align}\label{17}
&\frac{\partial}{\partial x^i}u^{i*}(s)\\
&=2(T_1)^{-1}\left(\sum_{j\in\eta_i}w_{ij}+k_i\right)\exp\left\{-\sigma^i B^i(s)+\frac{1}{2}(\sigma^i)^2 s\right\}d\lambda^i(s)\notag\\
&\pm 2\biggr\{\left(\sum_{j\in\eta_i}w_{ij}+k_i\right)^2\biggr[1+2x^i(s)\exp\left\{-\sigma^i B^i(s)+\frac{1}{2}(\sigma^i)^2 s\right\} d\lambda^i(s)\biggr]^2\notag\\
&-64\exp\left\{-4\sigma^iB^i(s)+2(\sigma^i)^2 s\right\}(d\lambda^i(s))^3\notag\\
&\times\left[k_i\left[x^i(s)-x^i_0\right]\exp\left\{\sigma^i B^i(s)-\frac{1}{2}(\sigma^i)^2 s\right\}+d\lambda^i(s)+\frac{d\lambda^i(s)}{ds}-\a(s)d\lambda^i(s)\right] \biggr\}^{-3/2}\notag\\
&\times\exp\left\{-\sigma^i B^i(s)+\frac{1}{2}(\sigma^i)^2 s\right\}\notag\\
&\times \left[1+2x^i(s)\exp\left\{-\sigma^i B^i(s)+\frac{1}{2}(\sigma^i)^2 s\right\} d\lambda^i(s)\right]-64\exp\left\{-3\sigma^iB^i(s)+\frac{3}{2}(\sigma^i)^2 s\right\}(d\lambda^i(s))^3.
\end{align}
Since $T_1>0$, the sign of the above partial derivative depends on the terms on two sides of $\pm$. Furthermore, as the optimal strategy cannot be negative and the term after $\pm$ in Equation \eqref{16} is a negative dominant term, we ignore the $+$ sign. Moreover, assuming $w_{ij}=k_i=0$, Equation \eqref{17} yields,
\begin{align}\label{18}
\frac{\partial}{\partial x^i}u^{i*}(s)&=-2\left\{-64\exp\left\{-4\sigma^iB^i(s)+2(\sigma^i)^2 s\right\}(d\lambda^i(s))^3\left[d\lambda^i(s)+\frac{d\lambda^i(s)}{ds}-\a(s)d\lambda^i(s)\right]\right\}^{-3/2}\notag\\
&\hspace{1cm}\times\exp\left\{-\sigma^i B^i(s)+\frac{1}{2}(\sigma^i)^2 s\right\}\left[1+2x^i(s)\exp\left\{-\sigma^i B^i(s)+\frac{1}{2}(\sigma^i)^2 s\right\} d\lambda^i(s)\right]\notag\\
&\hspace{2cm}-64\exp\left\{-3\sigma^iB^i(s)+\frac{3}{2}(\sigma^i)^2 s\right\}(d\lambda^i(s))^3>0.
\end{align}
Above result is true for all positive values of $w_{ij}$ and $k_i$. This implies that an agent's opinion positively influence their optimal strategy.

\medskip

{\bf Case II.}\\
Consider the opinion of agent $i$ is less influential than agent $j$ or, $\left[x^i(s)-x^j(s)\right]<0$. By construction $\phi_\theta\left(||x^i(s)-x^j(s)||\right)=0$. The terms $T_1$ and $T_2$ take the same value as in Case I. The other term is
\begin{align*}
T_3&=4\exp\left\{-2\sigma^i B^i(s)+(\sigma^i)^2 s\right\}d\lambda^i(s)\left[\left[\sum_{j\in\eta_i}w_{ij}\left[x^i(s)-x^j(s)\right]+k_i\left[x^i(s)-x^i_0\right]\right]\right.\\
&\hspace{1cm}\left.\times\exp\left\{\sigma^i B^i(s)-\frac{1}{2}(\sigma^i)^2 s\right\}+d\lambda^i(s)+\frac{d\lambda^i(s)}{ds}-\a(s)d\lambda^i\right].
\end{align*} 
Hence,
\begin{align}\label{19}
&\frac{\partial}{\partial x^i}u^{i*}(s)\notag\\
&=2(T_1)^{-1}\left(\sum_{j\in\eta_i}w_{ij}+k_i\right)\exp\left\{-\sigma^i B^i(s)+\frac{1}{2}(\sigma^i)^2 s\right\}d\lambda^i(s)\notag\\
&\pm 2\biggr\{\left(\sum_{j\in\eta_i}w_{ij}+k_i\right)^2\left[1+2x^i(s)\exp\left\{-\sigma^i B^i(s)+\frac{1}{2}(\sigma^i)^2 s\right\} d\lambda^i(s)\right]^2\notag\\
&-64\exp\left\{-4\sigma^iB^i(s)+2(\sigma^i)^2 s\right\}(d\lambda^i(s))^3\notag\\
&\times\biggr[\left[\sum_{j\in\eta_i}w_{ij}\left[x^i(s)-x^j(s)\right]+k_i\left[x^i(s)-x^i_0\right]\right]\exp\left\{\sigma^i B^i(s)-\frac{1}{2}(\sigma^i)^2 s\right\}\notag\\
&+d\lambda^i(s)+\frac{d\lambda^i(s)}{ds}-\a(s)d\lambda^i(s)\biggr] \biggr\}^{-3/2}\notag\\
&\times\exp\left\{-\sigma^i B^i(s)+\frac{1}{2}(\sigma^i)^2 s\right\} \left[1+2x^i(s)\exp\left\{-\sigma^i B^i(s)+\frac{1}{2}(\sigma^i)^2 s\right\} d\lambda^i(s)\right]\notag\\
&-64\exp\left\{-3\sigma^iB^i(s)+\frac{3}{2}(\sigma^i)^2 s\right\}(d\lambda^i(s))^3.
\end{align}
After assuming $w_{ij}=k_i=0$, the Equation \eqref{19} becomes Equation \eqref{18}. Therefore, $\frac{\partial}{\partial x^i}u^{i*}(s)>0$. This implies agent $i$'s opinion positively influence $u^{i*}(s)$ even agent $j$'s opinion is more influential in the society. Furthermore,
\begin{multline*}
\frac{\partial}{\partial x^j}u^{i*}(s)=-4\left\{\sum_{j\in\eta_i}w_{ij}\left[x^j(s)-x^i(s)\right]^{-3/2}\right\}\left(\sum_{j\in\eta_i}w_{ij}\right)\\
\times\exp\left\{-\frac{3}{2}\sigma^iB^i(s)+\frac{3}{4}(\sigma^i)^2s\right\}(d\lambda^i(s))^{3/2}<0.
\end{multline*}
The above equation shows a negative correlation between the $i^{th}$ agent's optimal strategy and the $j^{th}$ agent's opinion. This implies that as the opinion of the more influential $j^{th}$ agent becomes stronger, the $i^{th}$ agent becomes more hesitant to make a decision.

\section*{Conclusion}

This paper has addressed the estimation of an optimal opinion control strategy for a representative agent embedded in a stochastic McKean-Vlasov SDE, which models the dynamics of opinion evolution in a socially interacting population. By leveraging a Feynman-type path integral method with an integrating factor, we obtained a characterization of the agent's feedback control \( u^{i*}(s) \) that minimizes a social cost functional governed by McKean-Vlasov stochastic differential equations. Utilizing a modified Friedkin-Johnsen model, which captures both individual memory and peer influence in opinion formation, we were able to derive a closed-form expression for the agent’s optimal control. The mathematical derivation permitted an explicit analysis of the influence of the agent’s own opinion \( x^i(s) \) and that of their neighbors \( x^j(s) \) under different alignment conditions, specifically when \( x^i(s) = x^j(s) \) and when \( x^i(s) < x^j(s) \). Although the structure of the resulting expressions rendered it difficult to establish a universal directional dependence of \( u^{i*}(s) \) on the relative magnitude of these opinions, our findings demonstrate that the optimal control is positively correlated with the agent’s own opinion across scenarios, regardless of the influence exerted by others in the network.

The results contribute to the broader literature on stochastic control and mean-field interactions by providing both an analytical approach and theoretical insight into decentralized decision-making in socially interactive systems. However, several important avenues remain open for further exploration. A particularly promising direction involves relaxing the assumption that the diffusion coefficient is known. Since the existence and uniqueness of invariant measures in McKean-Vlasov SDEs can be highly sensitive to the noise amplitude, exploring the interplay between noise magnitude and long-run opinion distributions could yield deeper understanding of consensus and polarization phenomena \citep{yusuf2025predictive,yusuf2025prognostic}. Moreover, incorporating network heterogeneity and conducting numerical experiments under varying graph structures, such as scale-free, small-world, or modular networks would provide practical insights into how topology influences opinion optimization. Another compelling extension involves the development of fractional McKean-Vlasov models to capture memory effects and non-Markovian dynamics more accurately. In such a setting, the temporal dependencies in the evolution of opinions may reveal more nuanced relationships between an agent’s strategy and their own historical trajectory as well as those of their peers. Taken together, these extensions hold the potential to significantly enhance the descriptive power of stochastic opinion models and their applicability to real-world social systems.

\section*{Data Availability.}
The author declares that no data have been used in this paper.

\section*{Funding.}
This research received no external funding.

\section*{Appendix.}

\medskip

\subsection*{Proof of Lemma \ref{lem1}.}
From the definition of Wasserstein distance distance explained above consider $\pi\in\mathcal P(G\times G)$ such that $\gamma\in\mathcal P_2(G)$ and $\gamma'\in\mathcal P_2(G)$ are the time marginals. Let $\pi\in\mathcal P_2(G\times G)$ be temporarily fixed. Lebesgue dominated convergence theorem yields
\[
W_2(\gamma,\gamma')^2\leq\int|x^i(s_1,\omega)-x^i(s_2,\omega)|^2\pi(d\omega),
\]
for all $s_1,s_2\in[0,t]$ and $\omega\in G$. This implies $[0,t]\ni s_2\hookrightarrow\gamma'\in\mathcal P_2(G)$ is continuous for Wasserstein distance $W_2$. Since $x_0^\in G$ is given for agent $i$, after replacing $\mathbb P_{(x^i)}$ by $\gamma'$ Equation (\ref{1}) becomes,
\begin{align}\label{2}
dx^i(s)&=\mu^i[s,x^i(s),\gamma',u^i(s)]ds+\sigma^i[s,x^i(s),\gamma',u^i(s)]dB^i(s),
\end{align}
with random coefficients $\mu^i$ and $\sigma^i$. Then Theorem 1.2 of \cite{carmona2016} implies the above equation has a unique strong solution denoted by $\bm x_\pi^i=\{x_\pi^i(s)\}_{s=0}^t$. It is important to note that the probability law corresponding to $\bm x_\pi^i$ is of order 2. Define a mapping $\Xi$ so that
\[
\mathcal P_2(G)\ni\pi\hookrightarrow \Xi(\pi)=\mathbb P_{(\bm x_\pi^i)}\in\mathcal P_2(G).
\]
A process $\bm x^i=\{\bm x^i(s)\}_{s=0}^t$ with $\E\sup_{s\in[0,t]}|x^i(s)|^2<\infty$ is a solution of a stochastic differential equation iff the probability law is a fixed point on $Xi$. In the rest of the proof it will be shown that indeed, the mapping $\Xi$ has a unique fixed point which is sufficient for existence and uniqueness  of the solution to the Equation (\ref{2}). Consider $\pi,\pi'\in\mathcal P_2(G)$. As $\bm x_\pi^i$ and $\bm x_{\pi'}^i$ have same initial condition $x_0^i$, Doob's maximal inequality with Assumption \ref{as1} imply
\begin{multline*}
\E\left\{\sup_{s_1\in[0,s_2]}\left|x_\pi^i(s_1)-x_{\pi'}^i(s_1)\right|^2\right\}\\
\leq 2\E\left\{\sup_{s_1\in[0,s_2]}\left|\int_0^{s_1}[\mu^i(\nu,x_\pi^i(\nu),\pi_\nu,u_\pi^i(\nu))-\mu^i(\nu,x_{\pi'}^{i}(\nu),{\pi'}_\nu,u_{\pi'}^i(\nu))]d\nu\right|^2\right\}\\
+2\E\left\{\sup_{s_1\in[0,s_2]}\left|\int_0^{s_1}[\sigma^i(\nu,x_\pi^i(\nu),\pi_\nu,u_\pi^i(\nu))-\sigma^i(\nu,x_{\pi'}^{i}(\nu),{\pi'}_\nu,u_{\pi'}^i(\nu))]dB^i(\nu)\right|^2\right\}\\
\leq \hat c(1+t)\left[\int_0^{s_2}\E\left\{\sup_{\nu\in[0,s_1]}|x_\pi^i(\nu)-x_{\pi'}^i(\nu)|^2ds_1\right\}+\int_0^{s_2}\E\left\{\sup_{\nu\in[0,s_1]}|u_\pi^i(\nu)-u_{\pi'}^i(\nu)|^2ds_1\right\}\right.\\
\left.+\int_0^{s_2}W_2(\pi_{s_1},\pi'_{s_1})ds_1+\E\left\{\int_0^{s_2}|\sigma^i(\nu,x_\pi^i(\nu),\pi_\nu,u_\pi^i(\nu))-\sigma^i(\nu,x_{\pi'}^{i}(\nu),{\pi'}_\nu,u_{\pi'}^i(\nu))|^2d\nu\right\}\right]\\
\leq\hat ct\left[\int_0^{s_2}\E\left\{\sup_{\nu\in[0,s_1]}|x_\pi^i(\nu)-x_{\pi'}^i(\nu)|^2ds_1\right\}+\int_0^{s_2}\E\left\{\sup_{\nu\in[0,s_1]}|u_\pi^i(\nu)-u_{\pi'}^i(\nu)|^2ds_1\right\}\right.\\\left.+\int_0^{s_2}W_2(\pi_{s_1},\pi'_{s_1})ds_1\right].
\end{multline*}
Gronwall-Bellman inequality implies
\begin{equation}\label{4}
\E\left\{\sup_{s_1\in[0,s_2]}|x_\pi^i(s_1)-x_{\pi'}^i(s_1)|^2\right\}\leq\hat ct\exp(\hat ct)\int_0^{s_2}W_2(\pi_{s_1},\pi'_{s_1})ds_1.
\end{equation}
Since
\[
W_2\left(\Xi(\pi),\Xi(\pi')\right)^2\leq \E\left\{\sup_{s_1\in[0,s_2]}|x_\pi^i(s_1)-x_{\pi'}^i(s_1)|^2\right\},
\]
and
\[
W_2(\pi_{s_1},\pi'_{s_1})\leq W_2(s_1)(\pi,\pi'),
\]
Equation (\ref{4}) yields
\[
W_2(s_2)\left(\Xi(\pi),\Xi(\pi')\right)^2\leq \hat ct\exp(\hat ct)\int_0^{s_2}W_2(s_1)(\pi_{s_1},\pi'_{s_1})ds_1.
\]
After iterating the above inequality and denoting by $\Xi^\rho$ the $\rho^{th}$ composition of mapping $\Xi$ with itself yields
\begin{equation*}
W_2(t)\left(\Xi_\rho(\pi),\Xi_\rho(\pi')\right)^2\leq [\hat ct\exp(\hat ct)]^\rho\int_0^t\frac{(t-s_1)^{\rho-1}}{(\rho-1)!}W_2(s_1)(\pi_{s_1},\pi'_{s_1})ds_1\leq\frac{(\hat ct)^\rho}{\rho!}W_2(t)\left(\pi,\pi'\right)^2.
\end{equation*}
For a large value of $\rho$ the mapping $\Xi^\rho$ shows strict contraction. Hence, $\Xi$ admits a unique fixed point. This completes the proof. $\square$

\medskip

\subsection*{Proof of Lemma \ref{lem2}.}

\medskip

For agent $i$ assume $\mathcal V^{i\epsilon}(s)=(1/\epsilon)[x^{i\epsilon}(s)-x^i(s)]-\mathcal V^i(s)$. Therefore, $\mathcal V^{i\epsilon}(0)=0$, and 
\[
x^{i\epsilon}(s)=x^i(s)+\epsilon\left[\mathcal V^i(s)+\mathcal V^{i\epsilon}(s)\right].
\]
Hence,
\begin{multline}\label{m0}
d \mathcal V^{i\epsilon}(s)=\left\{\frac{1}{\epsilon}\left[\mu^i\left(s,x^{i\epsilon}(s),\mathbb P_{(x^{i\epsilon})},u^{i\epsilon}(s)\right)-\mu^i\left(s,x^i(s),\mathbb P_{(x^i)},u^i(s)\right)\right]\right.\\
\left.-\frac{\partial}{\partial x^i}\mu^i\left(s,x^i(s),\mathbb P_{(x^i)},u^i(s)\right)\mathcal V^i(s)-\zeta\left(s,\mathbb P_{(x^i,\mathcal V^i)}\right)-\frac{\partial}{\partial u^i}\mu^i\left(s,x^i(s),\mathbb P_{(x^i)},u^i(s)\right)\right\}ds\\
+\left\{\frac{1}{\epsilon}\left[\sigma^i\left(s,x^{i\epsilon}(s),\mathbb P_{x^{i\epsilon}},u^{i\epsilon}(s)\right)-\sigma^i\left(s,x^i(s),\mathbb P_{x^i},u^i(s)\right)\right]\right.\\
\left. -\frac{\partial}{\partial x^i}\sigma^i\left(s,x^i(s),\mathbb P_{(x^i)},u^i(s)\right)\mathcal V^i(s)-\zeta\left(s,\mathbb P_{(x^i,\mathcal V^i)}\right)-\frac{\partial}{\partial u^i}\sigma^i\left(s,x^i(s),\mathbb P_{(x^i)},u^i(s)\right)\right\}dB^i(s), 
\end{multline}
where
\[
\zeta(.)=\tilde\E\left\{\frac{\partial}{\partial \gamma}\mu^i(s,x^i(s),\mathbb P_{(x^i)},u^i(s))(\hat x^i(s)).\hat {\mathcal V}^i(s)\bigg |_{x^i=x^i(s),u^i=u^i(s)}\right\},
\]
and
\[
\hat\zeta(.)=\tilde\E\left\{\frac{\partial}{\partial \gamma}\sigma^i(s,x^i(s),\mathbb P_{(x^i)},u^i(s))(\hat x^i(s)).\hat {\mathcal V}^i(s)\bigg |_{x^i=x^i(s),u^i=u^i(s)}\right\}.
\]
After putting the values of $\zeta(.)$ and $\hat\zeta(.)$, Equation (\ref{m0}) yields
\begin{multline}\label{m1}
d \mathcal V^{i\epsilon}(s)=\left\{\frac{1}{\epsilon}\left[\mu^i\left(s,x^{i\epsilon}(s),\mathbb P_{(x^{i\epsilon})},u^{i\epsilon}(s)\right)-\mu^i\left(s,x^i(s),\mathbb P_{(x^i)},u^i(s)\right)\right]-\frac{\partial}{\partial x^i}\mu^i\left(s,x^i(s),\mathbb P_{(x^i)},u^i(s)\right)\mathcal V^i(s)\right.\\
\left.-\tilde\E\left\{\frac{\partial}{\partial \gamma}\mu^i(s,x^i(s),\mathbb P_{(x^i)},u^i(s))(\hat x^i(s)).\hat {\mathcal V}^i(s)\bigg |_{x^i=x^i(s),u^i=u^i(s)}\right\}-\frac{\partial}{\partial u^i}\mu^i\left(s,x^i(s),\mathbb P_{(x^i)},u^i(s)\right)\right\}ds\\
+\left\{\frac{1}{\epsilon}\left[\sigma^i\left(s,x^{i\epsilon}(s),\mathbb P_{(x^{i\epsilon})},u^{i\epsilon}(s)\right)-\sigma^i\left(s,x^i(s),\mathbb P_{(x^i)},u^i(s)\right)\right]-\frac{\partial}{\partial x^i}\sigma^i\left(s,x^i(s),\mathbb P_{(x^i)},u^i(s)\right)\mathcal V^i(s)\right.\\
\left.-\tilde\E\left\{\frac{\partial}{\partial \gamma}\sigma^i(s,x^i(s),\mathbb P_{(x^i)},u^i(s))(\hat x^i(s)).\hat {\mathcal V}^i(s)\bigg |_{x^i=x^i(s),u^i=u^i(s)}\right\}-\frac{\partial}{\partial u^i}\sigma^i\left(s,x^i(s),\mathbb P_{(x^i)},u^i(s)\right)\right\}dB^i(s)\\
=Ads+BdB^i(s).
\end{multline}

For each continuous time point $s\in[0,t]$ there exists a $\epsilon>0$ small enough such that 
\begin{multline}\label{m2}
\frac{1}{\epsilon}\left[\mu^i\left(s,x^{i\epsilon}(s),\mathbb P_{(x^{i\epsilon})},u^{i\epsilon}(s)\right)-\mu^i\left(s,x^i(s),\mathbb P_{(x^i)},u^i(s)\right)\right]\\
=\frac{1}{\epsilon}\left[\mu^i\left(s,x^{i}(s)+\epsilon\left(\mathcal V^i(s)+\mathcal V^{i\epsilon}(s)\right),\mathbb P_{\left(x^{i}(s)+\epsilon\left(\mathcal V^i(s)+\mathcal V^{i\epsilon}(s)\right)\right)},u^{i}(s)+\epsilon v^i(s)\right)-\mu^i\left(s,x^i(s),\mathbb P_{(x^i)},u^i(s)\right)\right]\\
=\int_0^1\frac{\partial}{\partial x^i}\mu^i\left(s,x^{i}(s)+\epsilon\be\left(\mathcal V^i(s)+\mathcal V^{i\epsilon}(s)\right),\mathbb P_{\left(x^{i}(s)+\epsilon\be\left(\mathcal V^i(s)+\mathcal V^{i\epsilon}(s)\right)\right)},u^{i}(s)+\epsilon\be v^i(s)\right)\left(\mathcal V^i(s)+\mathcal V^{i\epsilon}(s)\right)d\be\\
+\int_0^1\tilde\E\biggr\{\frac{\partial}{\partial \gamma}\mu^i\left(s,x^{i}(s)+\epsilon\be\left(\mathcal V^i(s)+\mathcal V^{i\epsilon}(s)\right),\mathbb P_{\left(x^{i}(s)+\epsilon\be\left(\mathcal V^i(s)+\mathcal V^{i\epsilon}(s)\right)\right)},u^{i}(s)+\epsilon\be v^i(s)\right)\\
\times\left[\hat x^{i}(s)+\epsilon\be\left(\hat{\mathcal V}^i(s)+\hat{\mathcal V}^{i\epsilon}(s)\right)\right]\left[\hat{\mathcal V}^i(s)+\hat{\mathcal V}^{i\epsilon}(s)\right]\biggr\}d\be\\
+\int_0^1\frac{\partial}{\partial u^i}\mu^i\left(s,x^{i}(s)+\epsilon\be\left(\mathcal V^i(s)+\mathcal V^{i\epsilon}(s)\right),\mathbb P_{\left(x^{i}(s)+\epsilon\be\left(\mathcal V^i(s)+\mathcal V^{i\epsilon}(s)\right)\right)},u^{i}(s)+\epsilon\be v^i(s)\right)\mathcal V^i(s)d\be.
\end{multline}
To get rid of notational complicacy define $x_\be^{i\epsilon}(s):=x^{i}(s)+\epsilon\be\left(\mathcal V^i(s)+\mathcal V^{i\epsilon}(s)\right)$, $\hat x_\be^{i\epsilon}(s):=\hat x^{i}(s)+\epsilon\be\left(\hat{\mathcal V}^i(s)+\hat{\mathcal V}^{i\epsilon}(s)\right)$, and $u_\be^{i\epsilon}(s):=u^{i}(s)+\epsilon\be v^i(s)$. Calculating ``ds" term of Equation (\ref{m1}) implies
\begin{multline*}
A=\int_0^1\frac{\partial}{\partial x^i}\mu^i\left(s,x_\be^{i\epsilon}(s),\mathbb P_{\left(x_\be^{i\epsilon}(s)\right)},u_\be^{i\epsilon}(s)\right)\mathcal V^{i\epsilon}(s)d\be\\
+\int_0^1\tilde\E\left\{\frac{\partial}{\partial \gamma}\mu^i\left(s,x_\be^{i\epsilon}(s),\mathbb P_{\left(x_\be^{i\epsilon}(s)\right)},u_\be^{i\epsilon}(s)\right)\right\}\left(\hat x_\be^{i\epsilon}(s)\hat{\mathcal V}^{i\epsilon}(s)\right)d\be\\
+\int_0^1\left[\frac{\partial}{\partial x^i}\mu^i\left(s,x_\be^{i\epsilon}(s),\mathbb P_{\left(x_\be^{i\epsilon}(s)\right)},u_\be^{i\epsilon}(s)\right)-\frac{\partial}{\partial x^i}\mu^i\left(s,x^i(s),\mathbb P_{(x^i)},u^i(s)\right)\right]\mathcal V^i(s) d\be\\
+\int_0^1\left[\tilde\E\left\{\frac{\partial}{\partial \gamma}\mu^i\left(s,x_\be^{i\epsilon}(s),\mathbb P_{\left(x_\be^{i\epsilon}(s)\right)},u_\be^{i\epsilon}(s)\right)-\frac{\partial}{\partial x^i}\mu^i\left(s,x^i(s),\mathbb P_{(x^i)},u^i(s)\right)\right\}\right]\left(\hat x_\be^{i\epsilon}(s)\hat{\mathcal V}^{i\epsilon}(s)\right)d\be\\
+\int_0^1\left[\frac{\partial}{\partial \gamma}\mu^i\left(s,x_\be^{i\epsilon}(s),\mathbb P_{\left(x_\be^{i\epsilon}(s)\right)},u_\be^{i\epsilon}(s)\right)-\frac{\partial}{\partial \gamma}\mu^i\left(s,x^i(s),\mathbb P_{(x^i)},u^i(s)\right)\right]\mathcal V^i(s)d\be\\
=\int_0^1\frac{\partial}{\partial x^i}\mu^i\left(s,x_\be^{i\epsilon}(s),\mathbb P_{\left(x_\be^{i\epsilon}(s)\right)},u_\be^{i\epsilon}(s)\right)\mathcal V^{i\epsilon}(s)d\be\\
+\int_0^1\tilde\E\left\{\frac{\partial}{\partial \gamma}\mu^i\left(s,x_\be^{i\epsilon}(s),\mathbb P_{\left(x_\be^{i\epsilon}(s)\right)},u_\be^{i\epsilon}(s)\right)\right\}\left(\hat x_\be^{i\epsilon}(s)\hat{\mathcal V}^{i\epsilon}(s)\right)d\be+I_1+I_2+I_3.
\end{multline*}
For $\epsilon\ra 0$, the integral terms $I_1$, $I_2$ and $I_3$ converges to zero in $\mathcal H^{2,\tilde m}([0,t]\times G)$. Moreover, for a finite constant $c>0$
\begin{multline*}
\E\left\{\int_0^t|I_1|^2ds\right\}=\E\left\{\int_0^t\left|\int_0^1\left[\frac{\partial}{\partial x^i}\mu^i\left(s,x_\be^{i\epsilon}(s),\mathbb P_{\left(x_\be^{i\epsilon}(s)\right)},u_\be^{i\epsilon}(s)\right)\right.\right.\right.\\
\left.\left.\left.-\frac{\partial}{\partial x^i}\mu^i\left(s,x^i(s),\mathbb P_{(x^i)},u^i(s)\right)\right]\mathcal V^i(s) d\be\right|^2ds\right\}\\
\leq \E\left\{\int_0^t\left|\int_0^1\left[\frac{\partial}{\partial x^i}\mu^i\left(s,x_\be^{i\epsilon}(s),\mathbb P_{\left(x_\be^{i\epsilon}(s)\right)},u_\be^{i\epsilon}(s)\right)-\frac{\partial}{\partial x^i}\mu^i\left(s,x^i(s),\mathbb P_{(x^i)},u^i(s)\right)\right]\right|^2\left|\mathcal V^i(s) \right|^2d\be ds\right\}\\
\leq c\E\left\{\int_0^t\int_0^1(\epsilon\be)^2\left[\left|\mathcal V^{i\epsilon}(s)-\mathcal V^i(s)\right|^2+|v^i(s)|^2\right]\left|\mathcal V^i(s)\right|^2d\be ds\right\}\\
\leq c\left[\int_0^t\int_0^1\E\left\{\left|\epsilon\be\left[\mathcal V^{i\epsilon}(s)-\mathcal V^i(s)\right]\right|^4d\be ds\right\}\right]^{1/2}+\left[\E\left\{|\mathcal V^i(s)|^4ds\right\}\right]^{1/2}\\
+c\left[\int_0^t\int_0^1\E\left\{|\epsilon\be v^i(s)|d\be ds\right\}\right]^{1/2}\left[\E\left\{|\mathcal V^i(s)|^4ds\right\}\right]^{1/2},
\end{multline*}
which converges to $0$ as $\epsilon\ra 0$ for all the above finite expectations. Similar argument goes to $I_2$ and $I_3$. To control the quadratic variation of the ``B" term in Equation (\ref{m1}) \emph{Burkholder-Davis-Gundy} can be used instead of \emph{Jensen's} inequality. Then 
\begin{align*}
\E\left\{\sup_{s\in[0,t]}|\mathcal V^{i\epsilon}(s)|^2ds\right\}&\leq c\left[\int_0^t\E\left\{\sup_{s_1\in[0,s]}|\mathcal V^{i\epsilon}(s_1)|^2\right\}ds+\int_0^t\sup_{s_1\in[0,s]}\left|\E\left\{\mathcal V^{i\epsilon}(s_1)\right\}\right|^2ds\right]+b_\epsilon\\
&\leq c\int_0^t\E\left\{\sup_{s_1\in[0,s]}|\mathcal V^{i\epsilon}(s_1)|^2\right\}ds+b_\epsilon,
\end{align*}
where $\lim_{\epsilon\ra 0} b_\epsilon =0$. The desired result would be obtained by implementing \emph{Gronwall's} inequality. This completes the proof. $\square$

\medskip

\subsection*{Proof of Lemma \ref{lem3}}

\medskip

For agent $i$ and $\nu\in[s,s+\epsilon]$ define $\mathcal V^{i\delta}(\nu):=(1/\delta)[x^{i\delta}(\nu)-x^i(\nu)]-\mathcal V^i(\nu)$. Therefore, $\mathcal V^{i\delta}(0)=0$, and 
\[
x^{i\delta}(\nu)=x^i(\nu)+\delta\left[\mathcal V^i(\nu)+\mathcal V^{i\delta}(\nu)\right].
\]
Hence,
\begin{multline*}
\frac{\partial}{\partial \delta}L^i(s,\bm x,u^i)\bigg|_{\delta=0}=\lim_{\delta\searrow 0}\frac{1}{\delta}\E_s\biggr\{\int_s^{s+\epsilon}\biggr[L^i\left[\nu,x^{-i}(\nu),x^i(\nu)+\delta\left[\mathcal V^i(s)+\mathcal V^{i\delta}(s)\right],u^i(\nu)+\delta v^i(\nu)\right]\\
-L^i\left[\nu,\bm x(\nu),u^i(\nu)\right]\biggr]d\nu\biggr\}\\
=\lim_{\delta\searrow 0}\frac{1}{\delta}\E_s\biggr\{\int_s^{s+\epsilon}\int_0^1\left[\frac{d}{d\be}L^i\left[\nu,x^{-i}(\nu),x^i(\nu)+\delta\be\left[\mathcal V^i(s)+\mathcal V^{i\delta}(s)\right],u^i(\nu)+\delta\be v^i(\nu)\right]\right]d\be d\nu\biggr\}\\
=\lim_{\delta\searrow 0}\frac{1}{\delta}\E_s\biggr\{\int_s^{s+\epsilon}\int_0^1\left[\left[\mathcal V^i(\nu)+\mathcal V^{i\delta}(\nu)\right]\frac{\partial}{\partial x^i}L^i\left[\nu,x^{-i}(\nu),x^i(\nu)+\delta\be\left[\mathcal V^i(s)+\mathcal V^{i\delta}(s)\right],u^i(\nu)+\delta\be v^i(\nu)\right]\right.\\
\left.+v^i(\nu)\frac{\partial}{\partial u^i}L^i\left[\nu,x^{-i}(\nu),x^i(\nu)+\delta\be\left[\mathcal V^i(s)+\mathcal V^{i\delta}(s)\right],u^i(\nu)+\delta\be v^i(\nu)\right]\right]d\be d\nu\biggr\}\\
=\E_s\left\{\int_s^{s+\epsilon}\left[\mathcal V^i(\nu)\frac{\partial}{\partial x^i}L^i\left[\nu,\bm x(\nu),u^i(\nu)\right]+v^i(\nu)\frac{\partial}{\partial u^i}L^i\left[\nu,\bm x(\nu),u^i(\nu)\right]\right]d\nu\right\}.
\end{multline*}
Since for $\nu\in[s,s+\epsilon]$
\[
\frac{\partial}{\partial x^i}L^i\left[\nu,\bm x(\nu),u^i(\nu)\right]=\int_s^{s+\epsilon}\left\{\sum_{j\in\eta_i}w_{ij}\left[x^i(\nu)-x^j(\nu)\right]+k_i\left[x^i(\nu)-x^i_0\right]\right\}d\nu,
\]
and 
\[
\frac{\partial}{\partial u^i}L^i\left[\nu,\bm x(\nu),u^i(\nu)\right]=\int_s^{s+\epsilon}u^i(\nu)d\nu,
\]
then
\[
\frac{\partial}{\partial \delta}L^i(s,\bm x,u^i)\bigg|_{\delta=0}=\E_s\left\{\int_s^{s+\epsilon}\left[\mathcal V^i(\nu)\sum_{j\in\eta_i}\left(w_{ij}\left[x^i(\nu)-x^j(\nu)\right]+k_i\left[x^i(\nu)-x_0^i\right]\right)+u^i(\nu)v^i(\nu)\right]d\nu\right\}.
\]
This completes the proof. $\square$

\medskip

\subsection*{Proof of Lemma \ref{lem4}}

\medskip

For all $\omega\in\Omega$ the expectation of the independent copy is defined as
\begin{multline*}
\tilde\E\left\{\left[\partial_\gamma\tilde{\mu}^i((s,\tilde{\bm x},\bm x,\tilde\lambda^i_1,\tilde\lambda^i_2,\tilde u^i))+\partial_\gamma\tilde{\sigma}^i(s,\tilde{\bm x},\bm x,\tilde\lambda^i_1,\tilde\lambda^i_2,\tilde u^i)\right]ds+\partial_\gamma\tilde{L}^i(s,\tilde{\bm x},\bm x,\tilde\lambda^i_1,\tilde\lambda^i_2,\tilde u^i)\bigg|_{\bm x=\bm x(s)}\right\}\\
:=\int_\Omega\biggr[\partial_\gamma\tilde{\mu}^i(s,\omega,\tilde\omega,\tilde{\bm x},\bm x,\tilde\lambda^i_1(\omega),\tilde\lambda^i_1(\tilde\omega),\tilde\lambda^i_2(\omega),\tilde\lambda^i_2(\tilde\omega),\tilde u^i)\\
+\partial_\gamma\tilde{\sigma}^i\left(s,\omega,\tilde\omega,\tilde{\bm x},\bm x,\tilde\lambda^i_1(\omega),\tilde\lambda^i_1(\tilde\omega),\tilde\lambda^i_2(\omega),\tilde\lambda^i_2(\tilde\omega),\tilde u^i\right)\\
+\partial_\gamma\tilde{L}^i\left(s,\omega,\tilde\omega,\tilde{\bm x},\bm x,\tilde\lambda^i_1(\omega),\tilde\lambda^i_1(\tilde\omega),\tilde\lambda^i_2(\omega),\tilde\lambda^i_2(\tilde\omega),\tilde u^i\right)\bigg|_{\bm x=\bm x(s)}\biggr]\mathbb P(d\tilde\omega)
\end{multline*}
For a constant $\rho\in(0,\infty)$ define a norm
\[
\left\|(\lambda_1^i,\lambda_2^i)\right\|_\rho^2:=\E\left\{\int_0^t\left[|\lambda_1^i(s)|^2+|\lambda_2^i(s)|^2\right]\exp(\rho s)ds\right\}.
\]
Let $(\lambda_1^{i \#},\lambda_2^{i \#})\in\mathcal H^{2,\tilde m}$ be another set of adjoint processes. The by Proposition 2.2 in \cite{pardoux1990} and Theorem 2.2 in \cite{carmona2016} there exists a unique solution $(\bm\lambda_1^{i*},\bm\lambda_2^{i*})$ of the adjoint process
\begin{multline*}
d\lambda_1^i(s)=-\partial_x	\mathcal L^i(s,\bm x,\mathbb P_{\bm x},\lambda_1^{i \#},\lambda_2^{i \#},\lambda^i_1,\lambda^i_2,u^i)+\lambda^i_2(s)dB^i(s)\\-\tilde\E\left[\partial_\gamma\tilde{\mathcal L}^i(s,\tilde{\bm x},\mathbb P_{\tilde{\bm x}},\lambda_1^{i \#},\lambda_2^{i \#},\tilde\lambda^i_1,\tilde\lambda^i_2,\tilde u^i)\right][x^i(s)].
\end{multline*}
Since, the above adjoint process is a forward looking process, the linear part represented by \\ $\partial_x	\mathcal L^i(s,\bm x,\mathbb P_{\bm x},\lambda_1^{i \#},\lambda_2^{i \#},\lambda^i_1,\lambda^i_2,u^i)$ has a unique solution $\left(\bm\lambda_1^{i*},\bm\lambda_2^{i*}\right)$ since at time $0$ the agent $i$ only makes expectations based on the available information at that time. There exists a map $M$ such that $\left(\lambda_1^{i \#},\lambda_2^{i \#}\right)\hookrightarrow\left(\bm\lambda_1^{i*},\bm\lambda_2^{i*}\right)=M\left(\lambda_1^{i \#},\lambda_2^{i \#}\right)$ from $\mathcal H^{2,\tilde m}$ into itself. Since, $\lambda_1^i\in\mathcal H_{\bm\lambda_1}^{2,\tilde m}$, for an appropriate choice of $\rho$ is necessary to show the existence of a strict contraction in $M$. Suppose, two pairs $\left(\lambda_1^{i\# 1},\lambda_2^{i\# 1}\right)\in\mathcal H^{2,\tilde m}$ and $\left(\lambda_1^{i\# 2},\lambda_2^{i\# 2}\right)\in\mathcal H^{2,\tilde m}$. Let $\left(\lambda_1^{i 1},\lambda_2^{i 1}\right)=M\left(\lambda_1^{i\# 1},\lambda_2^{i\# 1}\right)$, $\left(\lambda_1^{i 2},\lambda_2^{i 2}\right)=M\left(\lambda_1^{i\# 2},\lambda_2^{i\# 2}\right)$, $\left(\tilde\lambda_1^{i\#},\tilde\lambda_2^{i\#}\right)=\left(\lambda_1^{i\# 2}-\lambda_1^{i\# 1},\lambda_2^{i\# 2}-\lambda_2^{i\# 1}\right)$, and $\left(\tilde\lambda_1^{i1}-\tilde\lambda_1^{i2}\right)=\left(\lambda_1^{i 2}-\lambda_1^{i 1},\lambda_2^{i 2}-\lambda_2^{i 1}\right)$. For all $s\in[0,t]$ implementing It\^o formula to $|\tilde\lambda_1^i(s)|^2\exp(\rho s)$ yields,
\begin{multline*}
|\tilde\lambda_1^i(s)|^2+\E\left\{\int_s^t\rho\exp[\rho(\nu-s)]\left|\tilde\lambda_1^i(\nu)\right|^2d\nu\bigg|\mathcal F_s\right\}+\E\left\{\int_s^t\rho\exp[\rho(\nu-s)]\left|\tilde\lambda_2^i(\nu)\right|^2d\nu\bigg|\mathcal F_s\right\}\\
=\E\biggr\{2\exp[\rho(\nu-s)]\left[\Xi\left(\nu,\tilde{\bm x},\bm x,\lambda^{i\#2}_1(\nu),\lambda^{i\#2}_2(\nu),\lambda^{i2}_1(\nu),\lambda^{i2}_2(\nu),u^i\right)\right.\\
\left.-\Xi\left(\nu,\tilde{\bm x},\bm x,\lambda^{i\#1}_1(\nu),\lambda^{i\#1}_2(\nu),\lambda^{i1}_1(\nu),\lambda^{i1}_2(\nu),u^i\right)\right]d\nu\bigg|\mathcal F_s\biggr\},
\end{multline*} 
where
\begin{multline*}
\Xi\left(\nu,\tilde{\bm x},\bm x,\lambda^{i\#2}_1(\nu),\lambda^{i\#2}_2(\nu),\lambda^{i2}_1(\nu),\lambda^{i2}_2(\nu),u^i\right)\\
=\partial_\gamma\tilde\mu^i\left(\nu,\tilde{\bm x},\bm x,\lambda^{i\#2}_1(\nu),\lambda^{i\#2}_2(\nu),\lambda^{i2}_1(\nu),\lambda^{i2}_2(\nu),u^i\right)+\partial_\gamma\tilde\sigma^i\left(\nu,\tilde{\bm x},\bm x,\lambda^{i\#2}_1(\nu),\lambda^{i\#2}_2(\nu),\lambda^{i2}_1(\nu),\lambda^{i2}_2(\nu),u^i\right)\\
+\partial_\gamma\tilde L^i\left(\nu,\tilde{\bm x},\bm x,\lambda^{i\#2}_1(\nu),\lambda^{i\#2}_2(\nu),\lambda^{i2}_1(\nu),\lambda^{i2}_2(\nu),u^i\right),
\end{multline*}
and
\begin{multline*}
\Xi\left(\nu,\tilde{\bm x},\bm x,\lambda^{i\#1}_1(\nu),\lambda^{i\#1}_2(\nu),\lambda^{i1}_1(\nu),\lambda^{i1}_2(\nu),u^i\right)\\
=\partial_\gamma\tilde\mu^i\left(\nu,\tilde{\bm x},\bm x,\lambda^{i\#1}_1(\nu),\lambda^{i\#1}_2(\nu),\lambda^{i1}_1(\nu),\lambda^{i1}_2(\nu),u^i\right)+\partial_\gamma\tilde\sigma^i\left(\nu,\tilde{\bm x},\bm x,\lambda^{i\#1}_1(\nu),\lambda^{i\#1}_2(\nu),\lambda^{i1}_1(\nu),\lambda^{i1}_2(\nu),u^i\right)\\
+\partial_\gamma\tilde L^i\left(\nu,\tilde{\bm x},\bm x,\lambda^{i\#1}_1(\nu),\lambda^{i\#1}_2(\nu),\lambda^{i1}_1(\nu),\lambda^{i1}_2(\nu),u^i\right).
\end{multline*}
Let $\rho=16a^1+4a+1$. Then by implementing condition (v) of Assumption \ref{as1} and $\mathcal F$-integrability of $\Xi$ in $\mathcal H^{2,\tilde m}$ yields
\begin{multline*}
\left(\frac{1}{2}\rho-2a-2a^2\right)\E\left\{\exp(\rho\nu)\left|\tilde\lambda_1^i(\nu)\right|^2d\nu\right\}+\frac{1}{2}\E\left\{\exp(\rho\nu)\left|\tilde\lambda_2^i(\nu)\right|^2d\nu\right\}\\
\leq\frac{4a^2}{\rho}\left[\E\left\{\exp(\rho\nu)\left|\tilde\lambda_1^{i\#}(\nu)\right|^2d\nu\right\}+\frac{1}{2}\E\left\{\exp(\rho\nu)\left|\tilde\lambda_2^{i\#}(\nu)\right|^2d\nu\right\}\right],
\end{multline*}
which implies
\[
\E\left\{\int_0^t\exp(\rho s)\left[\left|\tilde\lambda_1^i(s)\right|^2+\left|\tilde\lambda_2^i(s)\right|^2\right]ds\right\}\leq\frac{1}{2}\E\left\{\int_0^t\exp(\rho s)\left[\left|\tilde\lambda_1^{i\#}(s)\right|^2+\left|\tilde\lambda_2^{i\#}(s)\right|^2\right]ds\right\}.
\]
Hence, $\left\|\tilde\lambda_1^i,\tilde\lambda_2^i\right\|_\rho\leq 2^{-1/2}\left\|\tilde\lambda_1^{i\#},\tilde\lambda_2^{i\#}\right\|_\rho$. This completes the proof. $\square$

\medskip

\subsection*{Proof of Proposition \ref{p0}.}

\medskip

The Euclidean action function of the system can be represented as 
\begin{align}
\mathcal A_{0,t}^i(x^i)&=\int_0^t\E_s\bigg\{L^i(s,\mathbf x,u^i)ds+\bigg[x^i(s)-x_0^i-\mu^i\left[s,x^i,\mathbb P_{(x^i)},u^i\right]ds-\sigma^i\left[s,x^i,\mathbb P_{(x^i)},u^i\right]\bigg]d\lambda^i(s)\bigg\},\notag
\end{align}
where $E_s$ is the conditional expectation on opinion $x^i(s)$ at the beginning of time $s$. For all $\varepsilon>0$, and the normalizing constant $L_\varepsilon^i>0$ , define a transitional function in small time interval as
\begin{align}\label{w16}
\Psi_{s,s+\varepsilon}^i(x^i)&:=\frac{1}{L_\varepsilon^i} \int_{\mathbb{R}} \exp\biggr\{-\varepsilon  \mathcal{A}_{s,s+\varepsilon}^i(x)\biggr\} \Psi_s^i(x^i) dx^i(s),
\end{align}	
for $\epsilon\downarrow 0$ and $\Psi_s^i(x^i)$ is the value of the transition function at time $s$ and opinion $x^i(s)$ with the initial condition $\Psi_0^i(x^i)=\Psi_0^i$ for all $i\in N$.

For continuous time interval $[s,\tau]$ where $\tau=s+\varepsilon$  the stochastic Lagrangian can be represented as,
\begin{align}\label{action}
\mathcal{A}_{s,\tau}^i(x)&= \int_{s}^{\tau}\ \E_s\ \biggr\{ L^i[\nu,\mathbf x(\nu),x_0^i,u^i(\nu)]\ d\nu+\bigg[x^i(\nu)-x_0^i-\mu^i\left[\nu,x^i,\mathbb P_{(x^i)},u^i\right]d\nu\notag\\
&\hspace{2cm}-\sigma^i\left[\nu,x^i,\mathbb P_{(x^i)},u^i\right]dB^i(\nu)\bigg]d\lambda^i(\nu) \biggr\},
\end{align}
with the constant initial condition $x^i(0)=x_0^i$.	This conditional expectation is valid when the control $u^i(\nu)$ of agent $i$'s opinion dynamics is determined at time $\nu$ such that all $n$-agents' opinions $\mathbf x(\nu)$ are given \citep{chow1996}. The evolution takes place as the action function is stationary. Therefore, the conditional expectation with respect to time only depends on the expectation of initial time point of interval $[s,\tau]$.

Fubini's Theorem implies,
\begin{align}\label{action5}
\mathcal{A}_{s,\tau}^i(x^i)&= \E_s\ \bigg\{ \int_{s}^{\tau}\ L^i[\nu,\mathbf x(\nu),x_0^i,u^i(\nu)]\ d\nu+\big[x^i(\nu)-x_0^i-\mu^i\left[\nu,x^i,\mathbb P_{(x^i)},u^i\right]d\nu\notag\\
&\hspace{2cm}-\sigma^i\left[\nu,x^i,\mathbb P_{(x^i)},u^i\right]dB^i(\nu)\big]d\lambda^i(\nu) \bigg\}.
\end{align}
By It\^o's Theorem there exists a function $h^i[\nu,x^i(\nu)]\in C^2([0,\infty)\times\mathbb{R})$ such that  $Y^i(\nu)=h^i[\nu,x^i(\nu)]$ where $Y^i(\nu)$ is an It\^o process \cite{oksendal2003}. Assuming 
\[
h^i[\nu+\Delta \nu,x^i(\nu)+\Delta x^i(\nu)]= x^i(\nu)-x_0^i-\mu^i\left[\nu,x^i(\nu),\mathbb P_{(x^i)},u^i(\nu)\right]d\nu-\sigma^i\left[\nu,x^i(\nu),\mathbb P_{(x^i)},u^i(\nu)\right]dB^i(\nu),
\]
Equation (\ref{action5}) implies,
\begin{align}\label{action6}
\mathcal{A}_{s,\tau}^i(x^i)&=\E_s \bigg\{ \int_{s}^{\tau}\ g^i[\nu,\mathbf x(\nu),u^i(\nu)]\ d\nu+ h^i\left[\nu+\Delta \nu,x^i(\nu)+\Delta x^i(\nu)\right]d\lambda^i(\nu)\bigg\}.
\end{align}
It\^o's Lemma implies,
\begin{align}\label{action7}
\varepsilon\mathcal{A}_{s,\tau}^i(x^i)&= \E_s \bigg\{\varepsilon L^i[s,\mathbf x(s),u^i(s)]+ \varepsilon h^i[s,x^i(s)]d\lambda^i(s)+ \varepsilon h_s^i[s,x^i(s)]d\lambda^i(s) \notag\\
&\hspace{.5cm}+\varepsilon h_x^i[s,x^i(s)]\mu^i\left[s,x^i(s),\mathbb P_{(x^i)},u^i(s)\right]d\lambda^i(s)\notag\\
&\hspace{1cm} +\varepsilon h_x^i[s,x^i(s)]\sigma^i\left[s,x^i(s),\mathbb P_{(x^i)},u^i(s)\right]d\lambda^i(s) dB^i(s)\notag\\
&\hspace{1.5cm}+\mbox{$\frac{1}{2}$}\varepsilon\left(\sigma^{i}\left[s,x^i(s),\mathbb P_{(x^i)},u^i(s)\right]\right)^2h_{xx}^i[s,x^i(s)]d\lambda^i(s)+o(\varepsilon)\bigg\},
\end{align}
where $h_s^i=\frac{\partial}{\partial s} h^i$, $h_x^i=\frac{\partial}{\partial x^i} h^i$ and $h_{xx}^i=\frac{\partial^2}{\partial (x^i)^2} h^i$, and we use the condition $[d x^i(s)]^2\approx\varepsilon$ with $d x^i(s)\approx\varepsilon\mu^i[s,x^i(s),u^i(s)]+\sigma^i[s,x^i(s),u^i(s)]dB^i(s)$. We use It\^o's Lemma and a similar approximation to approximate the integral. With $\varepsilon\downarrow 0$, dividing throughout by  $\varepsilon$ and taking the conditional expectation yields,
\begin{align}\label{action8}
\varepsilon\mathcal{A}_{s,\tau}^i(x^i)&= \E_s \bigg\{\varepsilon L^i[s,\mathbf x(s),u^i(s)]+\varepsilon h^i[s,x^i(s)]d\lambda^i(s)+ \varepsilon h_s^i[s,x^i(s)]d\lambda^i(s)\notag\\
&\hspace{.5cm}+\varepsilon h_x^i[s,x^i(s)]\mu^i\left[s,x^i(s),\mathbb P_{(x^i)},u^i(s)\right]d\lambda^i(s)\notag\\
&\hspace{1cm}+\mbox{$\frac{1}{2}$}\varepsilon\sigma^{2i}\left[s,x^i(s),\mathbb P_{(x^i)},u^i(s)\right]h_{xx}^i[s,x^i(s)]d\lambda^i(s)+o(1)\bigg\},
\end{align}
since $\E_s[dB^i(s)]=0$ and $\E_s[o(\varepsilon)]/\varepsilon\ra 0$ for all $\varepsilon\downarrow 0$ with the initial condition $x_0^i$. For $\varepsilon\downarrow 0$ denote a transition function at $s$ as $\Psi_s^i(x^i)$ for all $i\in N$. Hence, using Equation (\ref{w16}), the transition function yields
\begin{multline}\label{action9}
\Psi_{s,\tau}^i(x^i)=\frac{1}{L_\epsilon^i}\int_{\mathbb{R}} \exp\biggr\{-\varepsilon \big[L^i[s,\mathbf x(s),u^i(s)]+h^i[s,x^i(s)]d\lambda^i(s)\\
 +h_s^i[s,x^i(s)]d\lambda^i(s) +h_x^i[s,x^i(s)]\mu^i\left[s,x^i(s),\mathbb P_{(x^i)},u^i(s)\right]d\lambda^i(s)\\
+\mbox{$\frac{1}{2}$}\left(\sigma^{i}\left[s,x^i(s),\mathbb P_{(x^i)},u^i(s)\right]\right)^2h_{xx}^i[s,x^i(s)]d\lambda^i(s)\big]\biggr\} \Psi_s^i(x) dx^i(s)+o(\varepsilon^{1/2}).
\end{multline}
Since $\varepsilon\downarrow 0$, first order Taylor series expansion on the left hand side of Equation (\ref{action9}) yields
\begin{multline}\label{action10}
\Psi_{is}(x^i)+\varepsilon  \frac{\partial \Psi_{is}(x^i) }{\partial s}+o(\varepsilon)=\frac{1}{L_\varepsilon^i}\int_{\mathbb{R}} \exp\biggr\{-\varepsilon \big[L^i[s,\mathbf x(s),u^i(s)]+h^i[s,x^i(s)]d\lambda^i(s) \\
+h_s^i[s,x^i(s)]d\lambda^i(s)+h_x^i[s,x^i(s)]\mu^i\left[s,x^i(s),\mathbb P_{(x^i)},u^i(s)\right]d\lambda^i(s)\\
+\mbox{$\frac{1}{2}$}\left(\sigma^{i}\left[s,x^i(s),\mathbb P_{(x^i)},u^i(s)\right]\right)^2h_{xx}^i[s,x^i(s)]d\lambda^i(s)\big]\biggr\} \Psi_s^i(x) dx^i(s)+o(\varepsilon^{1/2}).
\end{multline}
For fixed $s$ and $\tau$ let $x^i(s)-x^i(\tau)=\xi^i$ so that $x^i(s)=x^i(\tau)+\xi^i$. When $\xi^i$ is not around zero, for a positive number $\eta<\infty$ we assume $|\xi^i|\leq\sqrt{\frac{\eta\varepsilon}{x^i(s)}}$ so that for $\varepsilon\downarrow 0$, $\xi^i$ takes even smaller values and agent $i$'s opinion $0<x^i(s)\leq\eta\varepsilon/(\xi^i)^2$. Therefore,
\begin{multline*}
\Psi_{is}(x^i)+\varepsilon\frac{\partial \Psi_{is}(x^i)}{\partial s}=\frac{1}{L_\epsilon^i}\int_{\mathbb{R}} \left[\Psi_{is}(x^i)+\xi^i\frac{\partial \Psi_{is}(x^i)}{\partial x^i}+o(\epsilon)\right]\\
\times \exp\biggr\{-\varepsilon \big[L^i[s,\mathbf x(s),u^i(s)]+h^i[s,x^i(s)]d\lambda^i(s)+h_x^i[s,x^i(s)]\mu^i\left[s,x^i(s),\mathbb P_{(x^i)},u^i(s)\right]d\lambda^i(s)\\
+\mbox{$\frac{1}{2}$}\left(\sigma^{i}\left[s,x^i(s),\mathbb P_{(x^i)},u^i(s)\right]\right)^2h_{xx}^i[s,x^i(s)]d\lambda^i(s)\big]\biggr\} d\xi^i+o(\varepsilon^{1/2}).
\end{multline*}
Before solving for Gaussian integral of the each term of the right hand side of the above Equation define a $C^2$ function 
\begin{align*}
f^i[s,\bm\xi,\lambda^i(s),\gamma,u^i(s)]&=L^i[s,\mathbf x(s)+\bm{\xi},u^i(s)]+h^i[s,x^i(s)+\xi^i]d\lambda^i(s) +h_s^i[s,x^i(s)+\xi^i]d\lambda^i(s)\notag\\
&\hspace{.25cm}+h_x^i[s,x^i(s)+\xi^i]\mu^i\left[s,x^i(s)+\xi,\mathbb P_{(x^i+\xi)},u^i(s)\right]d\lambda^i(s)\\
&\hspace{1cm}+\mbox{$\frac{1}{2}$}\sigma^{2i}\left[s,x^i(s)+\xi,\mathbb P_{(x^i+\xi)},u^i(s)\right]h_{xx}^i[s,x^i(s)+\xi^i]d\lambda^i(s)+o(1),
\end{align*}
where $\bm{\xi}$ is a vector of all $n$-agents' $\xi^i$'s.  Hence,
\begin{align}\label{action13}
\Psi_{is}(x^i)+\varepsilon \frac{\partial \Psi_{is}(x^i) }{\partial s}&=\Psi_{is}(x^i)\frac{1}{L_\epsilon^i}\int_{\mathbb{R}}\exp\left\{-\varepsilon f^i[s,\bm\xi,\lambda^i(s),\gamma,u^i(s)] \right\}d\xi^i\notag\\
&+\frac{\partial \Psi_{is}(x^i)}{\partial x^i}\frac{1}{L_\epsilon^i}\int_{\mathbb{R}}\xi^i\exp\left\{-\varepsilon f^i[s,\bm\xi,\lambda^i(s),\gamma,u^i(s)] \right\}d\xi^i+o(\varepsilon^{1/2}).
\end{align}
After taking $\varepsilon\downarrow 0$, $\Delta u\downarrow0$ and a Taylor series expansion with respect to $x^i$ of $f^i[s,\bm\xi,\lambda^i(s),\gamma,u^i(s)]$ yields, 
\begin{align*}
f^i[s,\bm\xi,\lambda^i(s),\gamma,u(s)]&=f^i[s,\mathbf x(\tau),\lambda^i(s),\gamma,u^i(s)]+f_x^i[s,\mathbf x(\tau),\lambda^i(s),\gamma,u^i(s)][\xi^i-x^i(\tau)]\notag\\
&\hspace{1cm}+\mbox{$\frac{1}{2}$}f_{xx}^i[s,\mathbf x(\tau),\lambda^i(s),\gamma,u^i(s)][\xi^i-x^i(\tau)]^2+o(\varepsilon).
\end{align*}
Define $y^i:=\xi^i-x^i(\tau)$ so that $ d\xi^i=dy^i$. First integral on the right hand side of Equation (\ref{action13}) becomes,
\begin{align}\label{action14}
&\int_{\mathbb{R}} \exp\big\{-\varepsilon f^i[s,\bm\xi,\lambda^i(s),\gamma,u^i(s)]\} d\xi^i\notag\\
&=\exp\big\{-\varepsilon f^i[s,\mathbf x(\tau),\lambda^i(s),\gamma,u^i(s)]\big\}\notag\\
&\hspace{1cm}\times\int_{\mathbb{R}} \exp\biggr\{-\varepsilon \biggr[f_x^i[s,\mathbf x(\tau),\lambda^i(s),\gamma,u^i(s)]y^i+\mbox{$\frac{1}{2}$}f_{xx}^i[s,\mathbf x(\tau),\lambda^i(s),\gamma,u^i(s)](y^i)^2\biggr]\biggr\} dy^i.
\end{align}
Assuming  $a^i=\frac{1}{2} f_{xx}^i[s,\mathbf x(\tau),\lambda^i(s),\gamma,u^i(s)]$ and $b^i=f_x^i[s,\mathbf x(\tau),\lambda^i(s),\gamma,u^i(s)]$ the argument of the exponential function in Equation (\ref{action14}) becomes,
\begin{align}\label{action15}
a^i(y^i)^2+b^iy^i&=a^i\left[(y^i)^2+\frac{b^i}{a^i}y^i\right]=a^i\left(y^i+\frac{b^i}{2a^i}y^i\right)^2-\frac{(b^i)^2}{4(a^i)^2}.
\end{align}
Therefore,
\begin{align}\label{action16}
&\exp\big\{-\varepsilon f^i[s,\mathbf x(\tau),\lambda^i(s),\gamma,u^i(s)]\big\}\int_{\mathbb{R}} \exp\big\{-\varepsilon [a^i(y^i)^2+b^iy^i]\big\}dy^i\notag\\
&=\exp\left\{\varepsilon \left[\frac{(b^i)^2}{4(a^i)^2}-f^i[s,\mathbf x(\tau),\lambda^i(s),\gamma,u^i(s)]\right]\right\}\int_{\mathbb{R}} \exp\left\{-\left[\varepsilon a^i\left(y^i+\frac{b^i}{2a^i}y^i\right)^2\right]\right\} dy^i\notag\\
&=\sqrt{\frac{\pi}{\varepsilon a^i}}\exp\left\{\varepsilon \left[\frac{(b^i)^2}{4(a^i)^2}-f^i[s,\mathbf x(\tau),\lambda^i(s),\gamma,u^i(s)]\right]\right\},
\end{align}
and
\begin{align}\label{action17}
&\Psi_{is}(x^i) \frac{1}{L_\varepsilon^i} \int_{\mathbb{R}} \exp\big\{-\varepsilon f^i[s,\bm\xi,\lambda^i(s),\gamma,u^i(s)]\} d\xi^i\notag\\ &=\Psi_{is}(x) \frac{1}{L_\varepsilon^i} \sqrt{\frac{\pi}{\varepsilon a^i}}\exp\left\{\varepsilon \left[\frac{(b^i)^2}{4(a^i)^2}-f^i[s,\mathbf x(\tau),\lambda^i(s),\gamma,u^i(s)]\right]\right\}. 
\end{align}
Substituting $\xi^i=x^i(\tau)+y^i$ second integrand of the right hand side of Equation (\ref{action13}) yields,
\begin{align}\label{action18}
& \int_{\mathbb{R}} \xi^i \exp\left[-\varepsilon \{f^i[s,\bm\xi,\lambda^i(s),\gamma,u^i(s)]\}\right] d\xi^i\notag\\
&=\exp\{-\varepsilon f^i[s,\mathbf x(\tau),\lambda^i(s),\gamma,u^i(s)]\}\int_{\mathbb{R}} [x^i(\tau)+y^i] \exp\left[-\varepsilon \left[a^i(y^i)^2+b^iy^i\right]\right] dy^i\notag\\
&=\exp\left\{\varepsilon \left[\frac{(b^i)^2}{4(a^i)^2}-f^i[s,\mathbf x(\tau),\lambda^i(s),\gamma,u^i(s)]\right]\right\} \biggr[x^i(\tau)\sqrt{\frac{\pi}{\varepsilon a^i}}\notag\\
&\hspace{1cm}+\int_{\mathbb{R}} y^i \exp\left\{-\varepsilon \left[a^i\left(y^i+\frac{b^i}{2a^i}y^i\right)^2\right]\right\} dy^i\biggr].
\end{align}
Substituting $k^i=y^i+b^i/(2a^i)$ in Equation (\ref{action18}) yields,
\begin{align}\label{action19}
&\exp\left\{\varepsilon \left[\frac{(b^i)^2}{4(a^i)^2}-f^i[s,\mathbf x(\tau),\lambda^i(s),\gamma,u^i(s)]\right]\right\} \biggr[x^i(\tau)\sqrt{\frac{\pi}{\varepsilon a^i}}+\int_{\mathbb{R}} \left(k^i-\frac{b^i}{2a^i}\right) \exp[-a^i\varepsilon (k^i)^2] dk^i\biggr]\notag\\
&=\exp\left\{\varepsilon \left[\frac{(b^i)^2}{4(a^i)^2}-f^i[s,\mathbf x(\tau),\lambda^i(s),\gamma,u^i(s)]\right]\right\} \biggr[x^i(\tau)-\frac{b^i}{2a^i}\biggr]\sqrt{\frac{\pi}{\varepsilon a^i}}.
\end{align}
Hence,
\begin{align}\label{action20}
&\frac{1}{L_\varepsilon^i}\frac{\partial \Psi_{is}(x^i)}{\partial x^i}\int_{\mathbb{R}} \xi^i \exp\left[-\varepsilon f[s,\bm\xi,\lambda^i(s),\gamma,u^i(s)]\right] d\xi^i\notag\\
&=\frac{1}{L_\varepsilon^i}\frac{\partial \Psi_{is}(x^i)}{\partial x^i} \exp\left\{\varepsilon \left[\frac{(b^i)^2}{4(a^i)^2}-f^i[s,\mathbf x(\tau),\lambda^i(s),\gamma,u^i(s)]\right]\right\} \biggr[x^i(\tau)-\frac{b^i}{2a^i}\biggr]\sqrt{\frac{\pi}{\varepsilon a^i}}.
\end{align}
Plugging in Equations (\ref{action17}), and (\ref{action20})  into Equation (\ref{action13}) implies,
\begin{align}\label{action24}
&\Psi_{is}(x^i)+\varepsilon \frac{\partial \Psi_{is}(x^i)}{\partial s}\notag\\
&=\frac{1}{L_\varepsilon^i} \sqrt{\frac{\pi}{\varepsilon a^i}}\Psi_{is}(x^i) \exp\left\{\varepsilon \left[\frac{(b^i)^2}{4(a^i)^2}-f^i[s,\mathbf x(\tau),\lambda^i(s),\gamma,u^i(s)]\right]\right\}\notag\\
&+\frac{1}{L_\varepsilon^i}\frac{\partial \Psi_{is}(x^i)}{\partial x^i} \sqrt{\frac{\pi}{\varepsilon a^i}} \exp\left\{\varepsilon \left[\frac{(b^i)^2}{4(a^i)^2}-f^i[s,\mathbf x(\tau),\lambda^i(s),\gamma,u^i(s)]\right]\right\} \biggr[x^i(\tau)-\frac{b^i}{2a^i}\biggr]+o(\varepsilon^{1/2}).
\end{align}
Let $f^i$ be in Schwartz space. This leads to derivatives are rapidly falling  and  further assuming $0<|b^i|\leq\eta\varepsilon$, $0<|a^i|\leq\mbox{$\frac{1}{2}$}[1-(\xi^i)^{-2}]^{-1}$ and $x^i(s)-x^i(\tau)=\xi^i$ yields,
\begin{align*}
x^i(\tau)-\frac{b^i}{2a^i}=x^i(s)-\xi^i-\frac{b^i}{2a^i}=x^i(s)-\frac{b^i}{2a^i},\ \forall\ \xi\downarrow 0,
\end{align*}
such that 
\begin{align*}
\bigg|x^i(s)-\frac{b^i}{2a^i}\bigg|=\bigg|\frac{\eta\varepsilon}{(\xi^i)^2}-\eta\varepsilon\left[1-\frac{1}{(\xi^i)^2}\right]\bigg|\leq\eta\varepsilon.
\end{align*}
Therefore, the Fokker-Plank-type Equation for agent $i$ is,
\begin{align}\label{action25.4}
\frac{\partial \Psi_{is}(x)}{\partial s}&=\left[\frac{(b^i)^2}{4(a^i)^2}-f^i[s,\mathbf x(\tau),\lambda^i(s),\gamma,u^i(s)]\right]\Psi_{is}(x).
\end{align}
Differentiating the Equation (\ref{action25.4}) with respect to $u^i$ yields an optimal control of agent $i$ under this opinion dynamics
\begin{align}\label{w18}
\left\{\frac{2f_x^i}{f_{xx}^i}\left[\frac{f_{xx}^if_{xu}^i-f_x^if_{xxu}^i}{(f_{xx}^i)^2}\right]-f_u^i\right\}\Psi_{is}(x)=0,
\end{align}
where $f_x^i=\frac{\partial}{\partial x^i} f^i$, $f_{xx}^i=\frac{\partial^2}{\partial (x^i)^2} f^i$, $f_{xu}^i=\frac{\partial^2}{\partial x^i\partial u^i} f^i$ and $f_{xxu}^i=\frac{\partial^3}{\partial (x^i)^2\partial u^i} f^i=0$. Thus, optimal feedback control of agent $i$ in stochastic opinion dynamics is represented as $u^{i*}(s,x^i)$ and is found by setting Equation (\ref{w18}) equal to zero. Hence, $u^{i*}(s,x^i)$ is the solution of the following Equation
\begin{align}\label{w21}
f_u^i (f_{xx}^i)^2=2f_x^i f_{xu}^i.\ \ \square
\end{align}

\bibliographystyle{apalike}
\bibliography{mybibfile}

\begin{thebibliography}{}

\bibitem[Acemo{\u{g}}lu et~al., 2013]{acemouglu2013opinion}
Acemo{\u{g}}lu, D., Como, G., Fagnani, F., and Ozdaglar, A. (2013).
\newblock Opinion fluctuations and disagreement in social networks.
\newblock {\em Mathematics of Operations Research}, 38(1):1--27.

\bibitem[Acemo{\u{g}}lu and Ozdaglar, 2011]{acemoglu2011opinion}
Acemo{\u{g}}lu, D. and Ozdaglar, A. (2011).
\newblock Opinion dynamics and learning in social networks.
\newblock {\em Dynamic Games and Applications}, 1:3--49.

\bibitem[Andersson and Djehiche, 2011]{andersson2011maximum}
Andersson, D. and Djehiche, B. (2011).
\newblock A maximum principle for sdes of mean-field type.
\newblock {\em Applied Mathematics \& Optimization}, 63:341--356.

\bibitem[Baaquie, 2007]{belal2007}
Baaquie, B.~E. (2007).
\newblock {\em Quantum finance: Path integrals and Hamiltonians for options and
  interest rates}.
\newblock Cambridge University Press.

\bibitem[Bauso et~al., 2016]{bauso2016opinion}
Bauso, D., Tembine, H., and Basar, T. (2016).
\newblock Opinion dynamics in social networks through mean-field games.
\newblock {\em SIAM Journal on Control and Optimization}, 54(6):3225--3257.

\bibitem[Bochner et~al., 1949]{bochner1949}
Bochner, S., Chandrasekharan, K., et~al. (1949).
\newblock {\em Fourier transforms}.
\newblock Number~19. Princeton University Press.

\bibitem[Brugna and Toscani, 2015]{brugna2015kinetic}
Brugna, C. and Toscani, G. (2015).
\newblock Kinetic models of opinion formation in the presence of personal
  conviction.
\newblock {\em Physical Review E}, 92(5):052818.

\bibitem[Bulls et~al., 2025]{bulls2025assessing}
Bulls, S.~E., Finn, E., Sykora, P., Lynch, V.~J., Pramanik, P., Glaberman, S.,
  and Chiari, Y. (2025).
\newblock Assessing cometchip technology for dna damage studies in non-model
  species: distinct uv-induced responses in turtles and mammals.
\newblock {\em BMC Research Notes}, 18(1):1--7.

\bibitem[Calvo-Armengol and Jackson, 2004]{calvo2004}
Calvo-Armengol, A. and Jackson, M.~O. (2004).
\newblock The effects of social networks on employment and inequality.
\newblock {\em American economic review}, 94(3):426--454.

\bibitem[Calv{\'o}-Armengol et~al., 2009]{calvo2009}
Calv{\'o}-Armengol, A., Patacchini, E., and Zenou, Y. (2009).
\newblock Peer effects and social networks in education.
\newblock {\em The Review of Economic Studies}, 76(4):1239--1267.

\bibitem[Cardaliaguet, 2012]{cardaliaguet2012}
Cardaliaguet, P. (2012).
\newblock Notes on mean field games. notes from p.l. lions' lecture at the
  coll\'ege de france.
\newblock Technical report, Technical report.

\bibitem[Carmona, 2016]{carmona2016}
Carmona, R. (2016).
\newblock {\em Lectures on BSDEs, stochastic control, and stochastic
  differential games with financial applications}.
\newblock SIAM.

\bibitem[Carmona and Delarue, 2015]{carmona2015}
Carmona, R. and Delarue, F. (2015).
\newblock Forward--backward stochastic differential equations and controlled
  mckean--vlasov dynamics.
\newblock {\em The Annals of Probability}, 43(5):2647--2700.

\bibitem[Carmona et~al., 2013]{carmona2013control}
Carmona, R., Delarue, F.~c., and Lachapelle, A. (2013).
\newblock Control of mckean--vlasov dynamics versus mean field games.
\newblock {\em Mathematics and Financial Economics}, 7:131--166.

\bibitem[Castellano et~al., 2009]{castellano2009statistical}
Castellano, C., Fortunato, S., and Loreto, V. (2009).
\newblock Statistical physics of social dynamics.
\newblock {\em Reviews of modern physics}, 81(2):591.

\bibitem[Chazelle et~al., 2017]{chazelle2017well}
Chazelle, B., Jiu, Q., Li, Q., and Wang, C. (2017).
\newblock Well-posedness of the limiting equation of a noisy consensus model in
  opinion dynamics.
\newblock {\em Journal of Differential Equations}, 263(1):365--397.

\bibitem[Chen et~al., 2022]{chen2022deep}
Chen, T., Wang, Z., and Theodorou, E.~A. (2022).
\newblock Deep graphic fbsdes for opinion dynamics stochastic control.
\newblock In {\em 2022 IEEE 61st Conference on Decision and Control (CDC)},
  pages 4652--4659. IEEE.

\bibitem[Chow, 1996]{chow1996}
Chow, G.~C. (1996).
\newblock The lagrange method of optimization with applications to portfolio
  and investment decisions.
\newblock {\em Journal of Economic Dynamics and Control}, 20(1-3):1--18.

\bibitem[Conley and Udry, 2010]{conley2010}
Conley, T.~G. and Udry, C.~R. (2010).
\newblock Learning about a new technology: Pineapple in ghana.
\newblock {\em American economic review}, 100(1):35--69.

\bibitem[Cosso et~al., 2023]{cosso2023}
Cosso, A., Gozzi, F., Kharroubi, I., Pham, H., and Rosestolato, M. (2023).
\newblock Optimal control of path-dependent mckean-vlasov sdes in infinite
  dimension.
\newblock {\em The Annals of Applied Probability (to appear)}.

\bibitem[Dasgupta et~al., 2023]{dasgupta2023frequent}
Dasgupta, S., Acharya, S., Khan, M.~A., Pramanik, P., Marbut, S.~M., Yunus, F.,
  Galeas, J.~N., Singh, S., Singh, A.~P., and Dasgupta, S. (2023).
\newblock Frequent loss of cacna1c, a calcium voltage-gated channel subunit is
  associated with lung adenocarcinoma progression and poor prognosis.
\newblock {\em Cancer Research}, 83(7\_Supplement):3318--3318.

\bibitem[Erd{\"o}s and R{\'e}nyi, 1959]{erdos1959}
Erd{\"o}s, P. and R{\'e}nyi, A. (1959).
\newblock “on random graphs,”.
\newblock {\em Publicationes Mathematicae}, 6:290--297.

\bibitem[Ewald and Nolan, 2024]{ewald2024adaptation}
Ewald, C.~O. and Nolan, C. (2024).
\newblock On the adaptation of the lagrange formalism to continuous time
  stochastic optimal control: A lagrange-chow redux.
\newblock {\em Journal of Economic Dynamics and Control}, 162:104855.

\bibitem[Feynman, 1948]{feynman1948}
Feynman, R.~P. (1948).
\newblock Space-time approach to non-relativistic quantum mechanics.
\newblock {\em Reviews of Modern Physics}, 20(2):367.

\bibitem[Feynman, 1949]{feynman1949}
Feynman, R.~P. (1949).
\newblock Space-time approach to quantum electrodynamics.
\newblock {\em Physical Review}, 76(6):769.

\bibitem[Friedkin and Johnsen, 1990]{friedkin1990social}
Friedkin, N.~E. and Johnsen, E.~C. (1990).
\newblock Social influence and opinions.
\newblock {\em Journal of mathematical sociology}, 15(3-4):193--206.

\bibitem[Fujiwara, 2017]{fujiwara2017}
Fujiwara, D. (2017).
\newblock {\em Rigorous time slicing approach to Feynman path integrals}.
\newblock Springer.

\bibitem[Hebb, 2005]{hebb2005}
Hebb, D.~O. (2005).
\newblock {\em The organization of behavior: A neuropsychological theory}.
\newblock Psychology Press.

\bibitem[Hertweck et~al., 2023]{hertweck2023clinicopathological}
Hertweck, K.~L., Vikramdeo, K.~S., Galeas, J.~N., Marbut, S.~M., Pramanik, P.,
  Yunus, F., Singh, S., Singh, A.~P., and Dasgupta, S. (2023).
\newblock Clinicopathological significance of unraveling mitochondrial pathway
  alterations in non-small-cell lung cancer.
\newblock {\em The FASEB Journal}, 37(7):e23018.

\bibitem[Hua et~al., 2019]{hua2019}
Hua, L., Polansky, A., and Pramanik, P. (2019).
\newblock Assessing bivariate tail non-exchangeable dependence.
\newblock {\em Statistics \& Probability Letters}, 155:108556.

\bibitem[Huang et~al., 2003]{huang2003individual}
Huang, M., Caines, P.~E., and Malham\'e, R.~P. (2003).
\newblock Individual and mass behaviour in large population stochastic wireless
  power control problems: centralized and nash equilibrium solutions.
\newblock In {\em 42nd IEEE International Conference on Decision and Control
  (IEEE Cat. No. 03CH37475)}, volume~1, pages 98--103. IEEE.

\bibitem[Kac, 1956]{kac1956foundations}
Kac, M. (1956).
\newblock Foundations of kinetic theory.
\newblock In {\em Proceedings of The third Berkeley symposium on mathematical
  statistics and probability}, volume~3, pages 171--197.

\bibitem[Kakkat et~al., 2023]{kakkat2023cardiovascular}
Kakkat, S., Pramanik, P., Singh, S., Singh, A.~P., Sarkar, C., and Chakroborty,
  D. (2023).
\newblock Cardiovascular complications in patients with prostate cancer:
  Potential molecular connections.
\newblock {\em International Journal of Molecular Sciences}, 24(8):6984.

\bibitem[Kappen, 2005]{kappen2005}
Kappen, H.~J. (2005).
\newblock Path integrals and symmetry breaking for optimal control theory.
\newblock {\em Journal of statistical mechanics: theory and experiment},
  2005(11):P11011.

\bibitem[Kappen, 2007]{kappen2007}
Kappen, H.~J. (2007).
\newblock An introduction to stochastic control theory, path integrals and
  reinforcement learning.
\newblock In {\em AIP conference proceedings}, volume 887, pages 149--181. AIP.

\bibitem[Khan et~al., 2023a]{khan2023myb0}
Khan, M., Acharya, S., Anand, S., Sameeta, F., Pramanik, P., Keel, C., Singh,
  S., Carter, J., Dasgupta, S., and Singh, A. (2023a).
\newblock Myb exhibits racially disparate expression, clinicopathologic
  association, and predictive potential for biochemical recurrence in prostate
  cancer, iscience. 26 (2023) 108487.

\bibitem[Khan et~al., 2023b]{khan2023myb}
Khan, M.~A., Acharya, S., Anand, S., Sameeta, F., Pramanik, P., Keel, C.,
  Singh, S., Carter, J.~E., Dasgupta, S., and Singh, A.~P. (2023b).
\newblock Myb exhibits racially disparate expression, clinicopathologic
  association, and predictive potential for biochemical recurrence in prostate
  cancer.
\newblock {\em Iscience}, 26(12).

\bibitem[Khan et~al., 2024]{khan2024mp60}
Khan, M.~A., Acharya, S., Kreitz, N., Anand, S., Sameeta, F., Pramanik, P.,
  Keel, C., Singh, S., Carter, J., Dasgupta, S., et~al. (2024).
\newblock Mp60-05 myb exhibits racially disparate expression and
  clinicopathologic association and is a promising predictor of biochemical
  recurrence in prostate cancer.
\newblock {\em Journal of Urology}, 211(5S):e1000.

\bibitem[Lasry and Lions, 2007]{lasry2007mean}
Lasry, J.-M. and Lions, P.-L. (2007).
\newblock Mean field games.
\newblock {\em Japanese journal of mathematics}, 2(1):229--260.

\bibitem[Love and Turner, 1993]{love1993note}
Love, C. and Turner, M. (1993).
\newblock Note on utilizing stochastic optimal control in aggregate production
  planning.
\newblock {\em European journal of operational research}, 65(2):199--206.

\bibitem[Lu et~al., 2021]{lu2021learning}
Lu, F., Maggioni, M., and Tang, S. (2021).
\newblock Learning interaction kernels in heterogeneous systems of agents from
  multiple trajectories.
\newblock {\em Journal of machine learning research}, 22(32):1--67.

\bibitem[Maki et~al., 2025]{maki2025new}
Maki, E., Glimm, T., Pramanik, P., Chiari, Y., and Kiskowski, M. (2025).
\newblock New approaches for capturing and estimating variation in complex
  animal color patterns from digital photographs: application to the eastern
  box turtle (terrapene carolina).
\newblock {\em PeerJ}, 13:e19690.

\bibitem[Marcet and Marimon, 2019]{marcet2019}
Marcet, A. and Marimon, R. (2019).
\newblock Recursive contracts.
\newblock {\em Econometrica}, 87(5):1589--1631.

\bibitem[Mas-Colell et~al., 1995]{mas1995}
Mas-Colell, A., Whinston, M.~D., Green, J.~R., et~al. (1995).
\newblock {\em Microeconomic theory}, volume~1.
\newblock Oxford university press New York.

\bibitem[McKean, 1967]{mckean1967propagation}
McKean, H.~P. (1967).
\newblock Propagation of chaos for a class of non-linear parabolic equations.
\newblock {\em Stochastic Differential Equations (Lecture Series in
  Differential Equations, Session 7, Catholic Univ., 1967)}, pages 41--57.

\bibitem[Moretti, 2011]{moretti2011}
Moretti, E. (2011).
\newblock Social learning and peer effects in consumption: Evidence from movie
  sales.
\newblock {\em The Review of Economic Studies}, 78(1):356--393.

\bibitem[Nakajima, 2007]{nakajima2007}
Nakajima, R. (2007).
\newblock Measuring peer effects on youth smoking behaviour.
\newblock {\em The Review of Economic Studies}, 74(3):897--935.

\bibitem[Niazi et~al., 2016]{niazi2016}
Niazi, M. U.~B., {\"O}zg{\"u}ler, A.~B., and Yildiz, A. (2016).
\newblock Consensus as a nash equilibrium of a dynamic game.
\newblock In {\em 2016 12th International Conference on Signal-Image Technology
  \& Internet-Based Systems (SITIS)}, pages 365--372. IEEE.

\bibitem[{\O}ksendal, 2003]{oksendal2003}
{\O}ksendal, B. (2003).
\newblock Stochastic differential equations.
\newblock In {\em Stochastic differential equations}, pages 65--84. Springer.

\bibitem[Pardoux and Peng, 1990]{pardoux1990}
Pardoux, E. and Peng, S. (1990).
\newblock Adapted solution of a backward stochastic differential equation.
\newblock {\em Systems \& control letters}, 14(1):55--61.

\bibitem[Polansky and Pramanik, 2021]{polansky2021motif}
Polansky, A.~M. and Pramanik, P. (2021).
\newblock A motif building process for simulating random networks.
\newblock {\em Computational Statistics \& Data Analysis}, 162:107263.

\bibitem[Pramanik, 2016]{pramanik2016}
Pramanik, P. (2016).
\newblock {\em Tail non-exchangeability}.
\newblock Northern Illinois University.

\bibitem[Pramanik, 2020a]{pramanik2020optimization}
Pramanik, P. (2020a).
\newblock Optimization of market stochastic dynamics.
\newblock {\em SN Operations Research Forum}, 1(4):31.

\bibitem[Pramanik, 2020b]{pramanik2020}
Pramanik, P. (2020b).
\newblock Optimization of market stochastic dynamics.
\newblock In {\em SN Operations Research Forum}, volume~1, pages 1--17.
  Springer.

\bibitem[Pramanik, 2021a]{pramanik2021}
Pramanik, P. (2021a).
\newblock Effects of water currents on fish migration through a feynman-type
  path integral approach under $\sqrt {8/3}$ liouville-like quantum gravity
  surfaces.
\newblock {\em Theory in Biosciences}, 140(2):205--223.

\bibitem[Pramanik, 2021b]{pramanik2021thesis}
Pramanik, P. (2021b).
\newblock {\em Optimization of Dynamic Objective Functions Using Path
  Integrals}.
\newblock PhD thesis, Northern Illinois University.

\bibitem[Pramanik, 2022a]{pramanik2022lock}
Pramanik, P. (2022a).
\newblock On lock-down control of a pandemic model.
\newblock {\em arXiv preprint arXiv:2206.04248}.

\bibitem[Pramanik, 2022b]{pramanik2022stochastic}
Pramanik, P. (2022b).
\newblock Stochastic control of a sir model with non-linear incidence rate
  through euclidean path integral.
\newblock {\em arXiv preprint arXiv:2209.13733}.

\bibitem[Pramanik, 2023a]{pramanik2021consensus}
Pramanik, P. (2023a).
\newblock Consensus as a nash equilibrium of a stochastic differential game.
\newblock {\em European Journal of Statistics}, 3:10--10.

\bibitem[Pramanik, 2023b]{pramanik2023cmbp}
Pramanik, P. (2023b).
\newblock Optimal lock-down intensity: A stochastic pandemic control approach
  of path integral.
\newblock {\em Computational and Mathematical Biophysics}, 11(1):20230110.

\bibitem[Pramanik, 2023c]{pramanik2023cont}
Pramanik, P. (2023c).
\newblock Path integral control in infectious disease modeling.
\newblock {\em arXiv preprint arXiv:2311.02113}.

\bibitem[Pramanik, 2023d]{pramanik2023path}
Pramanik, P. (2023d).
\newblock Path integral control of a stochastic multi-risk sir pandemic model.
\newblock {\em Theory in Biosciences}, pages 1--36.

\bibitem[Pramanik, 2024a]{pramanik2024dependence}
Pramanik, P. (2024a).
\newblock Dependence on tail copula.
\newblock {\em J}, 7(2):127--152.

\bibitem[Pramanik, 2024b]{pramanik2024estimation}
Pramanik, P. (2024b).
\newblock Estimation of optimal lock-down and vaccination rate of a stochastic
  sir model: A mathematical approach.
\newblock {\em European Journal of Statistics}, 4:3--3.

\bibitem[Pramanik, 2024c]{pramanik2024measuring}
Pramanik, P. (2024c).
\newblock Measuring asymmetric tails under copula distributions.
\newblock {\em European Journal of Statistics}, 4:7--7.

\bibitem[Pramanik, 2024d]{pramanik2024stochastic}
Pramanik, P. (2024d).
\newblock Stochastic control in determining a soccer player’s performance.
\newblock {\em J. Compr. Pure Appl. Math}, 2:111.

\bibitem[Pramanik, 2025a]{pramanik2025optimal}
Pramanik, P. (2025a).
\newblock An optimal level of stubbornness to win a soccer match.
\newblock {\em arXiv preprint arXiv:2501.18050}.

\bibitem[Pramanik, 2025b]{pramanik2025stubbornness}
Pramanik, P. (2025b).
\newblock Stubbornness as control in professional soccer games: A bppsde
  approach.
\newblock {\em Mathematics}, 13(3):475.

\bibitem[Pramanik et~al., 2024]{pramanik2024parametric}
Pramanik, P., Boone, E.~L., and Ghanam, R.~A. (2024).
\newblock Parametric estimation in fractional stochastic differential equation.
\newblock {\em Stats}, 7(3):745.

\bibitem[Pramanik and Dong, 2025a]{pramanik2025impact}
Pramanik, P. and Dong, L. (2025a).
\newblock Impact of random monetary shock: a keynesian case.
\newblock {\em arXiv preprint arXiv:2505.00800}.

\bibitem[Pramanik and Dong, 2025b]{pramanik2025strategic}
Pramanik, P. and Dong, L. (2025b).
\newblock Strategic complementarities due to monetary shock under sticky price.
\newblock {\em European Journal of Statistics}, 5:9--9.

\bibitem[Pramanik et~al., 2025a]{pramanik2025factors}
Pramanik, P., Graff, J., and Decaro, M. (2025a).
\newblock On factors influencing consumer preference in pipeline stages: an
  experiment.
\newblock {\em arXiv preprint arXiv:2501.03418}.

\bibitem[Pramanik et~al., 2025b]{pramanik2025strategies}
Pramanik, P., Graff, J., and Decaro, M. (2025b).
\newblock Strategies to increase pipeline status: A case study from eclinical
  data.
\newblock {\em European Journal of Statistics}, 5:3--3.

\bibitem[Pramanik and Maity, 2024]{pramanik2024bayes}
Pramanik, P. and Maity, A.~K. (2024).
\newblock Bayes factor of zero inflated models under jeffereys prior.
\newblock {\em arXiv preprint arXiv:2401.03649}.

\bibitem[Pramanik and Polansky, 2020]{pramanik2020motivation}
Pramanik, P. and Polansky, A.~M. (2020).
\newblock Motivation to run in one-day cricket.
\newblock {\em arXiv preprint arXiv:2001.11099}.

\bibitem[Pramanik and Polansky, 2021]{pramanik2021optimala}
Pramanik, P. and Polansky, A.~M. (2021).
\newblock Optimal estimation of brownian penalized regression coefficients.
\newblock {\em arXiv preprint arXiv:2107.02291}.

\bibitem[Pramanik and Polansky, 2023a]{pramanik2023optimization001}
Pramanik, P. and Polansky, A.~M. (2023a).
\newblock Optimization of a dynamic profit function using euclidean path
  integral.
\newblock {\em SN Business \& Economics}, 4(1):8.

\bibitem[Pramanik and Polansky, 2023b]{pramanik2021scoring}
Pramanik, P. and Polansky, A.~M. (2023b).
\newblock Scoring a goal optimally in a soccer game under liouville-like
  quantum gravity action.
\newblock {\em Operations Research Forum}, 4(3):66.

\bibitem[Pramanik and Polansky, 2023c]{pramanik2023semicooperation}
Pramanik, P. and Polansky, A.~M. (2023c).
\newblock Semicooperation under curved strategy spacetime.
\newblock {\em The Journal of Mathematical Sociology}, pages 1--35.

\bibitem[Pramanik and Polansky, 2024]{pramanik2024motivation}
Pramanik, P. and Polansky, A.~M. (2024).
\newblock Motivation to run in one-day cricket.
\newblock {\em Mathematics}, 12(17):2739.

\bibitem[Sharrock et~al., 2021]{sharrock2021parameter}
Sharrock, L., Kantas, N., Parpas, P., and Pavliotis, G.~A. (2021).
\newblock Parameter estimation for the mckean-vlasov stochastic differential
  equation.
\newblock {\em arXiv preprint arXiv:2106.13751}.

\bibitem[Sharrock et~al., 2023]{sharrock2023online}
Sharrock, L., Kantas, N., Parpas, P., and Pavliotis, G.~A. (2023).
\newblock Online parameter estimation for the mckean--vlasov stochastic
  differential equation.
\newblock {\em Stochastic Processes and their Applications}, 162:481--546.

\bibitem[Sheng, 2020]{sheng2020}
Sheng, S. (2020).
\newblock A structural econometric analysis of network formation games through
  subnetworks.
\newblock {\em Econometrica}, 88(5):1829--1858.

\bibitem[Snijders, 2002]{snijders2002}
Snijders, T.~A. (2002).
\newblock Markov chain monte carlo estimation of exponential random graph
  models.
\newblock {\em Journal of Social Structure}, 3(2):1--40.

\bibitem[Stella et~al., 2013]{stella2013opinion}
Stella, L., Bagagiolo, F., Bauso, D., and Como, G. (2013).
\newblock Opinion dynamics and stubbornness through mean-field games.
\newblock In {\em 52nd IEEE Conference on Decision and Control}, pages
  2519--2524. IEEE.

\bibitem[Sznitman, 1991]{sznitman1991topics}
Sznitman, A.-S. (1991).
\newblock Topics in propagation of chaos.
\newblock {\em Ecole d’{\'e}t{\'e} de probabilit{\'e}s de Saint-Flour
  XIX—1989}, 1464:165--251.

\bibitem[Theodorou et~al., 2010]{theodorou2010}
Theodorou, E., Buchli, J., and Schaal, S. (2010).
\newblock Reinforcement learning of motor skills in high dimensions: A path
  integral approach.
\newblock In {\em Robotics and Automation (ICRA), 2010 IEEE International
  Conference on}, pages 2397--2403. IEEE.

\bibitem[Theodorou, 2011]{theodorou2011}
Theodorou, E.~A. (2011).
\newblock {\em Iterative path integral stochastic optimal control: Theory and
  applications to motor control}.
\newblock University of Southern California.

\bibitem[Valdez and Pramanik, 2025a]{valdez2025association}
Valdez, I. and Pramanik, P. (2025a).
\newblock Association between obesity, race, and luminal subtypes of breast
  cancer.
\newblock {\em European Journal of Statistics}, 5:12--12.

\bibitem[Valdez and Pramanik, 2025b]{valdez2025exploring}
Valdez, I. and Pramanik, P. (2025b).
\newblock Exploring the interplay of adiposity, ethnicity, and hormone receptor
  profiles in breast cancer subtypes.
\newblock {\em arXiv preprint arXiv:2507.21348}.

\bibitem[Vikramdeo et~al., 2024a]{vikramdeo2024abstract}
Vikramdeo, K., Anand, S., Sudan, S., Pramanik, P., Singh, S., Godwin, A.,
  Singh, A., and Dasgupta, S. (2024a).
\newblock Abstract po3-16-05: Mitochondrial dna mutation detection in tumors
  and circulating extracellular vesicles of triple negative breast cancer
  patients for biomarker development.
\newblock {\em Cancer Research}, 84(9\_Supplement):PO3--16.

\bibitem[Vikramdeo et~al., 2024b]{vikramdeo2024mitochondrial}
Vikramdeo, K., Anand, S., Sudan, S., Pramanik, P., Singh, S., Godwin, A.,
  Singh, A., and Dasgupta, S. (2024b).
\newblock Mitochondrial dna mutation detection in tumors and circulating
  extracellular vesicles of triple negative breast cancer patients for
  biomarker development.
\newblock In {\em CANCER RESEARCH}, volume~84. AMER ASSOC CANCER RESEARCH 615
  CHESTNUT ST, 17TH FLOOR, PHILADELPHIA, PA~….

\bibitem[Vikramdeo et~al., 2023]{vikramdeo2023profiling}
Vikramdeo, K.~S., Anand, S., Sudan, S.~K., Pramanik, P., Singh, S., Godwin,
  A.~K., Singh, A.~P., and Dasgupta, S. (2023).
\newblock Profiling mitochondrial dna mutations in tumors and circulating
  extracellular vesicles of triple-negative breast cancer patients for
  potential biomarker development.
\newblock {\em FASEB BioAdvances}, 5(10):412.

\bibitem[Yusuf and Pramanik, 2025a]{yusuf2025predictive}
Yusuf, A. and Pramanik, P. (2025a).
\newblock Predictive significance of cd276/b7-h3 expression in baseline
  biopsies of advanced prostate carcinoma.
\newblock {\em arXiv preprint arXiv:2508.09373}.

\bibitem[Yusuf and Pramanik, 2025b]{yusuf2025prognostic}
Yusuf, A. and Pramanik, P. (2025b).
\newblock Prognostic role of b7-h3 (cd276) expression in initial biopsies of
  metastatic prostate cancer.
\newblock {\em Onco}, 5(3):38.

\end{thebibliography}
\end{document}